\documentclass[11pt]{amsart}

\usepackage{amssymb,amsmath,amsfonts,mathrsfs,enumerate}
\numberwithin{equation}{section}
\usepackage{hyperref}

\newtheorem{theorem}{Theorem}[section]

\newtheorem{lemma}[theorem]{Lemma}

\newtheorem{prop}[theorem]{Proposition}
\newtheorem{coro}[theorem]{Corollary}

\theoremstyle{definition}
\newtheorem{definition}[theorem]{Definition}

\newcommand{\Id}{{\bf 1}}
\newcommand{\SH}{{\rm SH}}
\newcommand{\SHm}{{\rm SH}_m(X,\omega)}

\newcommand{\capa}{{\rm Cap}}
\newcommand{\setdef}{\; | \;}
\newcommand{\Ephi}{\mathscr{E}_{\phi}}
\newcommand{\Em}{\mathscr{E}}
\newcommand{\Eone}{\mathscr{E}^1}

\hypersetup{
    bookmarks=true,         
    unicode=false,          
    pdftoolbar=true,       
    pdfmenubar=true,       
    pdffitwindow=false,    
    pdfstartview={FitH},   
    pdftitle={Complex Hessian equations},    
    pdfauthor={C.H. Lu, V.D. Nguyen},     
    colorlinks=true,       
   linkcolor=black,          
    citecolor=black,        
    filecolor=black,      
    urlcolor=black}           

\title[Complex Hessian equations with prescribed singularity]{Complex Hessian equations with prescribed singularity on compact K\"ahler manifolds}
   
\author{Chinh H. Lu}
\address{Universit\'e Paris Sud}
\email{\href{mailto:hoang-chinh.lu@u-psud.fr}{hoang-chinh.lu@u-psud.fr}}
\author{Van-Dong Nguyen}
\address{Ho Chi Minh city University of Education}
\email{\href{mailto:dongnv@hcmup.edu.vn}{dongnv@hcmup.edu.vn}}


\thanks{C.H. Lu is supported by the PEPS project JCJC 2019}
\date{\today}

\begin{document}

\begin{abstract}
Let $(X,\omega)$ be a compact K\"ahler manifold of dimension $n$ and fix $1\leq m\leq n$. We prove that the total mass of the complex Hessian measure of $\omega$-$m$-subharmonic functions is non-decreasing with respect to the singularity type.  We then solve complex Hessian equations with prescribed singularity, and  prove a Hodge index type inequality for positive currents. 
\end{abstract}
\maketitle

\tableofcontents

\section{Introduction}
Let  $(X,\omega)$ be a compact K\"ahler manifold of dimension $n$ and fix an integer $m$ such that $1\leq m\leq n$. For convenience we normalize $\omega$ such that $\int_X \omega^n=1$. 

In this paper we study complex Hessian equations of the form 
\begin{equation}\label{eq: Hes intro}
(\omega +dd^c u)^m \wedge \omega^{n-m} = \mu,
\end{equation}
where $\mu$ is a positive measure, and we want to solve the equation for $u$ in a given singularity class.

The case when $m=n$ (the Monge-Amp\`ere case)  has  numerous important applications in differential geometry, see \cite{Au78, Yau78, Kol98}, to only cite a few.
The complex Hessian equation appears in the study of the Fu-Yau equation  related to the Strominger system  \cite{PPZ17,PPZ18,PPZ19}. It is also motivated by the study of the Calabi problem for HKT-manifolds \cite{AV10}.  Its real counterpart, the real Hessian equation, was studied intensively with many interesting applications \cite{CNS85,T95,CW01}.

After several attempts \cite{Kok10}, \cite{Hou}, \cite{Jbi}, the existence of smooth solutions in the  smooth case (when $\mu = e^f \omega^n$, for some smooth function $f$) was solved \cite{DK17} by combining a Liouville type theorem for $m$-subharmonic functions  \cite{DK17} and a second order a priori estimate \cite{HMW}. This idea was recently used in  \cite{Sze18}, \cite{CP19} to solve the Dirichlet problem for complex Hessian equations on complex manifolds.   Degenerate solutions were studied in \cite{DK14,DK18}, \cite{GN18},  \cite{KN16}, \cite{Lu13, LN15}  and many others.

In \cite{LN15}, the authors have developed a global potential theory for $\omega$-$m$-subharmonic functions, solving \eqref{eq: Hes intro} in the full mass class $\mathcal{E}(X,\omega,m)$. This class consists of functions with very mild singularity, e.g. in case $n=m$, these have zero Lelong number everywhere.  In this paper we extend the study of \cite{LN15} to classes of $\omega$-$m$-sh functions with heavy singularities, inspired by \cite{DDL2,DDL4,DDL5}.  To do this, we first need a monotonicity result which is the first main result of this paper.

\begin{theorem}
	\label{thm: monotonicity intro}
	Assume that $u_1,...,u_m,v_1,...,v_m$ are $\omega$-$m$-sh functions on $X$ such that $u_p\leq v_p$, for all $p\in \{1,...,m\}$. Then 
	$$
	\int_X H_m(u_1,...,u_m) \leq \int_X H_m(v_1,...,v_m). 
	$$
\end{theorem}
Here $H_m(u_1,...,u_m):= (\omega+dd^c u_1) \wedge ... \wedge (\omega+dd^c u_m) \wedge \omega^{n-m}$ is the non-$m$-polar product; the relevant definitions will be given in Section \ref{sect: Preliminary}. 

For $n=m$, the above result was conjectured in \cite{BEGZ10} in the general context of big cohomology classes, and proved in \cite{WN19}. The monotonicity result in \cite{WN19} can also be proved  using geodesic rays \cite{DDL6}.  The approach of \cite{WN19} was recently used in \cite{Xia19} to prove an integration by parts formula.  
Our proof of Theorem \ref{thm: monotonicity intro} uses the monotonicity of the Hessian energy avoiding the geodesic notion which is not yet avaliable in the Hessian setting. 
\medskip

Having the monotonicity result and using recent techniques in \cite{DDL2,DDL4} we  study the complex Hessian equation with prescribed singularities.   
The second main result is the following: 

\begin{theorem}
	\label{thm: Hes eq intro}
	Assume that $\phi$ is a $\omega$-$m$-sh function   such that $P[\phi]=\phi$. Let $\mu$ be a non-$m$-polar positive measure  such that $\mu(X)=\int_X H_m(\phi)>0$.  Then there exists a unique $u\in \Ephi$ normalized by $\sup_X u=0$,  such that 
$
H_m(u)= \mu. 
$
\end{theorem}

The definition of the envelope $P[u]$, and the relative finite energy class $\Ephi$ will be given in Section \ref{sect: envelope}.  One can prove the uniqueness of solution by slightly modifying the proof of S. Dinew in the Mong-Amp\`ere case (see \cite{DiwJFA09,DL15}), which crucially uses the resolution of the equation. We propose in this paper an alternative proof using the fact that  the Hessian measure of the envelope is supported on the contact set.  To prove the existence of solutions we use the supersolution method of \cite{GLZJDG} as in \cite{DDL4}: we take the lower envelope of supersolutions. To do so, we need to bound the supersolutions from below.  This was done in \cite{DDL4} by establishing a relative $L^{\infty}$-estimate which is quite delicate in the Hessian setting  due to a lack of integrability of $\omega$-$m$-subharmonic functions. We overcome this by constructing $\omega$-$m$-subharmonic subextensions via a complete metric in the space $\mathscr{E}^1$, inspired by \cite{Dar15,Dar17IMRN,DDL3}. 

\medskip

Using the resolution of the complex Hessian equations with prescribed singularity we prove a Hodge-index type inequality for positive closed $(1,1)$-currents. 
	
\begin{theorem}\label{thm: log concave}
Let $u_j,j=1,...,m$ be $\omega$-$m$-subharmonic functions on $X$. Then 
$$
\int_X H_m(u_1,...,u_m) \geq   \prod_{k=1}^m \left (\int_X H_m(u_k) \right)^{1/m}. 
$$
\end{theorem}	
The above result  generalizes that of \cite{DDL4} which considers the case $m=n$, and \cite{Xiao} which considers smooth forms. Other directions  can also be explored to extend the above result to the case of  big cohomology classes. The proof of Theorem \ref{thm: log concave} is an obvious modification of the Monge-Amp\`ere case (see \cite{DDL2,DDL4}) given Theorem \ref{thm: monotonicity intro} and Theorem \ref{thm: Hes eq intro}.

\medskip

\noindent {\bf Organization of the paper.} In Section \ref{sect: Preliminary} we recall backgrounds on $\omega$-$m$-subharmonic functions and the complex Hessian operator. The relative potential theory adapted to the Hessian setting is discussed in Section \ref{sect: Relative PP}, where we prove Theorem \ref{thm: monotonicity intro} in Section \ref{sect: monotonicity} (Theorem \ref{thm: monotonicity}). We use the metric defined in Section \ref{sect: metric}  to establish the existence of solutions in Section \ref{sect: solutions}, where Theorem \ref{thm: Hes eq intro} is proved (Theorem \ref{thm: existence general}). The uniqueness is given a new proof  in Section \ref{sect: uniqueness}. Theorem \ref{thm: log concave} is proved in Section \ref{subsect: log-concave inequality}. 

\section{Backgrounds}\label{sect: Preliminary}

Let $(X,\omega)$ be  a compact K\"ahler manifold of dimension $n$, and fix an integer $m$ such that $1\leq m \leq n$. 

\subsection{$\omega$-$m$-subharmonic functions} In this section, we recall backgrounds on  $m$-subharmonic functions on domains as well as on compact  K\"ahler manifolds. Many properties of the complex Hessian operator can be proved by easy adaptations of the Monge-Amp\`ere case.   More details  on several classes of $m$-subharmonic functions  can be found in \cite{Li04,Bl05,SA12,DL15,Ch16,Lu15,LN15,DK14,DK17,NNC14,NNC12,GN18,EG17,NN13,LT19} and the references therein. 

Fix $\Omega$ an open subset of $\mathbb{C}^n$ and $\beta := dd^c \rho$ a K\"ahler form in $\Omega$ with smooth bounded potential. 

\begin{definition}
A function $u\in C^2(\Omega,\mathbb{R})$ is called  $m$-subharmonic ($m$-sh for short) with respect to $\beta$ if the following inequalities  hold in $\Omega$ : 
$$
(dd^c u)^k \wedge \beta^{n-k} \geq 0, \ \forall k \in \{1,...,m\}. 
$$	
\end{definition}

\begin{definition}
A function $u\in L^{1}(\Omega,\mathbb{R})$ is called $m$-subharmonic with respect to $\beta$ if  
\begin{enumerate}
	\item $u$ is upper semicontinuous in $\Omega$,
	\item $dd^c u \wedge dd^c u_2 \wedge ... \wedge dd^c u_m \wedge \beta^{n-m} \geq 0$, for all $u_2,...,u_m \in C^2(\Omega)$, $m$-sh with respect to $\beta$, 
	\item if $v\in L^1(\Omega)$ satisfies the above two conditions and $u=v$ a.e. in $\Omega$ then $u \leq v$.  
\end{enumerate}
\end{definition}
As observed by B{\l}ocki \cite{Bl05},  G{\aa}rding's inequality \cite{Gar59} ensures that the two definitions of  $m$-sh functions above coincide for smooth functions. 
\begin{definition}
A function $u\in L^1(X,\omega^n)$ is called $\omega$-$m$-subharmonic ($\omega$-$m$-sh for short)  if, locally in $\Omega \subset X$ where $\omega=dd^c \rho$,  $u+\rho$ is $m$-subharmonic with respect to $\omega$. 

The set of all $\omega$-$m$-sh functions on $X$ is denoted by $\SH_{m}(X,\omega)$. 
\end{definition}

The above definition depends heavily on the K\"ahler form $\omega$. This makes the smooth  approximation of $\omega$-$m$-subharmonic  functions quite complicated unless $\omega$ is flat.  Nevertheless, it was shown in  \cite{LN15}, \cite{KN16} using the viscosity theory and an approximation scheme of Berman \cite{Berm18}, and in \cite{Pl13}, \cite{HLP16} using the local smooth resolution, that the smooth approximation of $m$-subharmonic functions is possible.  As mentioned in \cite{HLP16}, the global approximation theorem in \cite{LN15} yields the local one. A direct proof of the local approximation property (which is also valid in the Hermitian setting) was given in \cite[Theorem 3.18]{GN18}. 
 \smallskip

 Given $u,v\in \SH_m(X,\omega)$, we say that $u$ is less singular than $v$ if there exists a constant $C$ such that $v\leq u+C$. We say that $u$ has the same singularity as $v$ if there exists a constant $C$ such that $u-C\leq v \leq u+C$. 
 \smallskip
 

In the flat case, B{\l}ocki proved in \cite{Bl05} that $m$-sh functions are in $L^p$ for any $p<n/(n-m)$, and conjectured that it holds for $p<nm/(n-m)$. Using the $L^{\infty}$ estimate due to S. Dinew and Ko{\l}odziej,  one can prove the same integrability property for $\omega$-$m$-sh functions, see \cite{DK14}, \cite[Corollary 6.7]{LN15}.

\subsection{Complex Hessian operator}

Given bounded $\omega$-$m$-sh functions $u_1$, $...,u_m$ the complex Hessian operator 
$$
H_m(u_1,...,u_m):= (\omega + dd^c u_1)  \wedge ... \wedge (\omega +dd^c u_m) \wedge  \omega^{n-m}
$$ 
is defined recursively by following Bedford-Taylor's seminal works \cite{BT76,BT82}. This gives a positive Borel measure and $H_m$ enjoys many nice convergence properties (see \cite{LN15},\cite{Lu13},\cite{GN18}).  When $u_1=...=u_m=u$ we simply denote the $m$-Hessian measure of $u$ by $H_m(u)$. 
\medskip

By plurifine locality   (see \cite{DK14,DK17,Lu13, LN15})  we have the following property: 
$$
\Id_U H_k(\max(u_1,v_1), ...,\max(u_m,v_m)) = \Id_U H_k(u_1,...,u_m),
$$
where  $u_1,..,u_m$, $v_1,...,v_m$ are bounded $\omega$-$m$-sh functions, and  $U:= \cap_{j=1}^m\{u_j>v_j\}$. 

For a Borel set $E\subset X$ we define 
$$
\capa_m(E) := \sup\left\{\int_E H_m(u) \setdef u \in \SH_m(X,\omega), \ -1\leq u \leq 0\right\}.
$$
A sequence of  functions $u_j$ converges in capacity to $u$ if for all $\varepsilon>0$, 
$$
\lim_{j\to +\infty}\capa_m (|u_j-u|>\varepsilon) =0. 
$$


Given $u_1,...,u_m \in \SH_m(X,\omega)$, not necessarily bounded,  and  $s>t$ we have
$$
{\bf 1}_{U^s} H_m(u_1^t,...,u_p^t) = {\bf 1}_{U^s} H_m(u_1^s,...,u_m^s),
$$
where $U^{s} := \cap_{p=1}^m \{u_p >-s\}$ and $u^s:= \max(u,-s)$.
It thus follows that the family of positive measures 
$
{\bf 1}_{U^s} H_m(u_1^s,...,u_m^s)
$
 is increasing in $s$, allowing to define 
 $$
 H_m(u_1,...,u_m) := \lim_{s\to +\infty}{\bf 1}_{U^s} H_m(u_1^s,...,u_m^s).
 $$
When $u_1=...=u_m=u$  we simply denote the Hessian  measure $H_m(u,u,...,u)$ by $H_m(u)$. An application of the Stokes theorem  gives
$$
0\leq \int_X H_m(u) \leq 1. 
$$

A Borel set $E$ is called $m$-polar (with respect to $\omega$) if there exists $u \in \SHm$ such that $E\subset \{u=-\infty\}$. 
\begin{lemma}
The positive measure $H_k(u)$ does not charge $m$-polar sets. 
\end{lemma}
\begin{proof}
If $v \in \SH_m(X,\omega)$ is bounded then $(2\omega +dd^c v)^m \wedge (2\omega)^{n-m}$ vanishes on $m$-polar sets (see \cite{Lu13,LN15}). Since 
$$
(2\omega +dd^c v)^m \wedge \omega^{n-m} = \sum_{k=0}^m \binom{m}{k} H_k(v),
$$
 it follows that  $H_k(v)$ also vanishes on $m$-polar sets for  $k =1,...,m$. 
Each $H_k(u_j)$ does not charge $m$-polar sets because $u_j:= \max(u,-j)$ is bounded. Since $H_k(u)$ is the strong limit of ${\bf 1}_{\{u>-j\}}H_k(u_j)$ it follows that $H_k(u)$ vanishes on $m$-polar sets. 
\end{proof}

\begin{definition}
	A Borel set $E\subset X$ is called quasi-open (quasi closed) if for each $\varepsilon>0$, there exists an open (closed) set $U$ such that 
	$$
	\capa_m( (E\setminus U) \cup (U\setminus E) ) <\varepsilon. 
	$$ 
\end{definition}

Since $\omega$-$m$-sh functions are quasi-continuous, see \cite{Lu13}, the sets of the form 
$$
\cap_{j=1}^N \{u_j>v_j\},
$$
where $u_j,v_j$ are $\omega$-$m$-sh functions, are quasi-open, while the corresponding sets with $\geq$ sign are quasi-closed. 

\begin{theorem}\label{thm: convergence quasi set}
Assume that $u_1^j,...,u_m^j$ are sequences of $\omega$-$m$-sh functions which are uniformly bounded. If $u_p^j$ converges in $m$-capacity to $u_p\in \SH_m(X,\omega)$, for all $p=1,...,m$, then 
$$
\liminf_{j} \int_E H_m(u_1^j,...,u_m^j) \geq \int_E H_m(u_1,...,u_m),
$$
for all quasi-open set $E$, and 
$$
\limsup_{j} \int_K H_m(u_1^j,...,u_m^j) \leq \int_K H_m(u_1,...,u_m),
$$
for all quasi-closed set $K$. 
\end{theorem}

The proof of the above theorem is an obvious modification of the Monge-Amp\`ere case, see \cite{GZbook}, \cite[Corollary 2.9]{DDL1}.

\medskip 

The following result, called the plurifine locality, will be used several times in this paper. 

\begin{lemma}\label{lem: Identity principle}
	Assume that $u_1,...,u_m,v_1,...,v_m$ are $\omega$-$m$-sh functions on $X$ and $\Omega\subset X$ is a quasi-open set such that $u_p=v_p$ on $\Omega$, for $p=1,...,m$. Then 
	$$
	{\bf 1}_{\Omega} H_m(u_1,...,u_m) = {\bf 1}_{\Omega} H_m(v_1,...,v_m). 
	$$
\end{lemma}
\begin{proof}
The proof for bounded functions is classical, see \cite[Corollary 4.3]{BT87} and the discussion in \cite[Section 1.2]{BEGZ10}. For convenience we repeat it here. 
For $\varepsilon>0$ set $w_p^{\varepsilon} := \max(u_p+\varepsilon,v_p)$, $w_p:= \max(u_p,v_p)$. Then $\Omega \subset \cap_{p=1}^m \{u_p+\varepsilon>v_p\}$, hence by the pluripotential maximum principle (see \cite[Theorem 3.14]{Lu13}, \cite[Theorem 3.27]{GZbook}), 
	$$
	{\bf 1}_{\Omega} H_m(w_1^{\varepsilon},...,w_m^{\varepsilon}) = {\bf 1}_{\Omega} H_m(u_1,...,u_m). 
	$$
	Since $\Omega$ is quasi open and the functions $u_p,v_p$ are uniformly bounded, letting $\varepsilon\to 0^+$   we obtain 
	$$
	{\bf 1}_{\Omega} H_m(w_1,...,w_m) \leq  {\bf 1}_{\Omega} H_m(u_1,...,u_m). 
	$$
	For a fixed compact subset $K\Subset \Omega$ we have
	$$
	{\bf 1}_{K} H_m(w_1^{\varepsilon},...,w_m^{\varepsilon}) = {\bf 1}_{K} H_m(u_1,...,u_m). 
	$$
	Letting $\varepsilon\to 0^+$ we arrive at 
	$$
	{\bf 1}_{K} H_m(w_1,...,w_m) \geq  {\bf 1}_{K} H_m(u_1,...,u_m). 
	$$
	Since the Hessian measure $H_m(u_1,...,u_m)$ is inner regular, we can conclude that 
	$$
	{\bf 1}_{\Omega} H_m(w_1,...,w_m)  =  {\bf 1}_{\Omega} H_m(u_1,...,u_m). 
	$$
	Changing the role of $u_p$ and $v_p$ we obtain the result for bounded functions.

For the general case we set $u_p^t := \max(u_p,-t)$, for $t>0$. From the previous step we have
	$$
	{\bf 1}_{\Omega} {\bf 1}_{U^t}H_m(u_1^t,...,u_m^t) ={\bf 1}_{\Omega} {\bf 1}_{V^t} H_m(v_1^t,...,v_m^t),	
	$$
where $U^t := \cap_{p=1}^m \{u_p>-t\}$, $V^t := \cap_{p=1}^m \{v_p>-t\}$. Now, we let $t\to +\infty$ to conclude the proof. 
\end{proof}

\begin{coro}\label{cor: plurifine}
	Assume that $u_1,...,u_m, v_1,...,v_m$ are $\omega$-$m$-sh on $X$. Then 
	$$
	{\bf 1}_{\Omega} H_m(\max(u_1,v_1)...,\max(u_m,v_m)) = {\bf 1}_{\Omega} H_m(u_1,...,u_m),
	$$
	where $\Omega:= \cap_{p=1}^m\{u_p>v_p\}$. 
\end{coro}

\begin{lemma}\label{lem: Dem inequality}
If $u,v \in \SH_m(X,\omega)$ then 
$$
H_m(\max(u,v)) \geq {\bf 1}_{\{u>v\}} H_m(u) + {\bf 1}_{\{u\leq v\}} H_m(v). 
$$
\end{lemma}
\begin{proof}
For $t>0$ set $u^t := \max(u,-t), v^t := \max(v,-t), \phi^t := \max(u^t,v^t)$. Then 
$$
H_m(\phi^t) \geq {\bf 1}_{\{u^t>v^t\}} H_m(u^t) + {\bf 1}_{\{u^t\leq v^t\}} H_m(v^t). 
$$
Multiplying both sides with ${\bf 1}_{U^t}$, where $U^t:=\{\min(u,v) >-t\}$, and using Lemma \ref{lem: Identity principle}, we obtain 
$$
{\bf 1}_{U^t}H_m(\phi) = {\bf 1}_{U^t} H_m(\phi^t) \geq  {\bf 1}_{U^t}{\bf 1}_{\{u>v\}} H_m(u) + {\bf 1}_{U^t}{\bf 1}_{\{u\leq v\}} H_m(v).
$$
Letting $t\to +\infty$ we arrive at the conclusion. 
\end{proof}

\begin{prop} \label{prop: viscosity}
If $u,v\in \SH_m(X,\omega)$ and  $u\leq v$, then 
$$
{\bf 1}_{\{u=v\}} H_m(u) \leq {\bf 1}_{\{u=v\}} H_m(v). 
$$
\end{prop}

Intuitively, $v$ can be thought of as an upper test function for $u$ on the contact set $\{u=v\}$, see \cite{EGZ11, Lu13JFA} for more details on the viscosity theory.  

\begin{proof}
We first assume that $u,v$ are bounded. For $\varepsilon>0$ set 
$
u_{\varepsilon} := \max(u,v-\varepsilon).
$
By Lemma \ref{lem: Dem inequality} we have 
$$
{\bf 1}_{\{u=v\}} H_m(u_{\varepsilon}) \geq {\bf 1}_{\{u=v\}}{\bf 1}_{\{u\geq v-\varepsilon\}} H_m(u) \geq {\bf 1}_{\{u=v\}} H_m(u).  
$$ 
Since the set $\{u=v\}$ is quasi-closed, and $u_{\varepsilon}$ is uniformly bounded, we can invoke Theorem \ref{thm: convergence quasi set} to get
$$
{\bf 1}_{\{u=v\}} H_m(v) \geq \limsup_{\varepsilon\to 0} {\bf 1}_{\{u=v\}} H_m(u_{\varepsilon}) \geq {\bf 1}_{\{u=v\}} H_m(u). 
$$
To treat the general case we set 
$$
u^t := \max(u,-t),\ v^t := \max(v,-t), \ U^t := \{u>-t\}.
$$ 
The first step gives 
$
{\bf 1}_{U^t}{\bf 1}_{\{u^t=v^t\}} H_m(v_t)\geq  {\bf 1}_{U^t}{\bf 1}_{\{u^t=v^t\}} H_m(u^t). 
$
Using Lemma \ref{lem: Identity principle} we  then have that
$
{\bf 1}_{U^t}{\bf 1}_{\{u=v\}} H_m(v)\geq  {\bf 1}_{U^t}{\bf 1}_{\{u=v\}} H_m(u). 
$
We finally let $t\to +\infty$ to arrive at the conclusion. 
\end{proof}

\begin{lemma}\label{lem: multilinearity}
	Assume that $u_1,...,u_m$ are $\omega$-$m$-sh on $X$ and $t_1,...,t_m \in [0,1]$ with $\sum_{p=1}^m t_p=1$. Then 
	$$
	H_m\left (\sum_{p=1}^m t_p u_p \right ) = \sum_{\sigma \in \Sigma} t_{\sigma(1)}...t_{\sigma(m)}H_m(u_{\sigma(1)},...,u_{\sigma(m)}),
	$$
	where $\Sigma$ is the set of all maps  $\sigma : \{1,...,m\} \rightarrow \{1,...,m\}$. 
\end{lemma}

\begin{proof}
	Fix $C>0$ and set 
	$$
	U^C:= \cap_{p=1}^m \{u_p>-C\}, \ \phi := \sum_{p=1}^m t_pu_p, \ \phi^C := \max(\phi,-C).
	$$ 
	Then $\phi >-C$ on $U^C$, hence by Lemma \ref{lem: Identity principle} we have 
	\begin{eqnarray*}
		 {\bf 1}_{U^C}  H_m(\phi)  &= & {\bf 1}_{U^C} H_m\left (\sum_{p=1}^m t_p u_p^C\right )\\
		 & =&   {\bf 1}_{U^C} \sum_{\sigma \in \Sigma} t_{\sigma(1)}...t_{\sigma(m)}H_m(u_{\sigma(1)}^C,...,u_{\sigma(m)}^C)\\
		 &=&  {\bf 1}_{U^C}  \sum_{\sigma \in \Sigma} t_{\sigma(1)}...t_{\sigma(m)}H_m(u_{\sigma(1)},...,u_{\sigma(m)}).
	\end{eqnarray*}
Letting $C\to +\infty$ we arrive at the conclusion. 
\end{proof}
\begin{lemma}[Mixed Hessian inequality]\label{lem: mixed Hes ineq}
Assume that $\mu$ is a  non-$m$-polar positive measure and $f_1,...,f_m$ are in $L^1(X,\mu)$. If $u_1,...,u_m \in \SH_m(X,\omega)$ satisfy $H_m(u_p) \geq f_p \mu$, $p=1,...,m$ then 
$$
H_m(u_1,...,u_m) \geq (f_1...f_m)^{1/m} \mu. 
$$
\end{lemma}

\begin{proof}
Having the mixed Hessian inequality for bounded $\omega$-$m$-sh functions \cite{DL15},  the proof of the lemma is identical to that of \cite[Proposition 1.11]{BEGZ10}.
\end{proof}

\subsection{Finite energy classes}
The class $\mathscr{E}(X,\omega,m)$ consists of functions $u\in \SH_m(X,\omega)$ such that $\int_X H_m(u)=1$. The class $\mathscr{E}^1(X,\omega,m)$ consists of $u \in \mathscr{E}(X,\omega,m)$ such that $\int_X |u|H_m(u) <+\infty$. 

To ease the notations, we will occasionally denote these classes by $\Em$, $\Eone$. 

The Hessian energy of $u \in \SH_m(X,\omega)\cap L^{\infty}(X)$ is defined by: 
$$
E_m(u) := \frac{1}{m+1} \sum_{k=0}^m \int_X u H_k(u). 
$$
When $(\omega,m)$ is fixed we will simply denote this functional by $E$. 

The following result is well-known in the Monge-Amp\`ere case and the proof can be adapted in an obvious way to the Hessian setting, see \cite{LN15}. 
\begin{prop}\label{prop: basic fact on energy} Suppose $u,v \in \SHm\cap L^{\infty}(X)$. The following hold:\\
\noindent (i) $ E(u)-E(v) = \frac{1}{m+1}\sum_{k=0}^n \int_X (u-v) \omega_{u}^k \wedge \omega_{v}^{m-k} \wedge \omega^{n-m}$.\\
\noindent (ii) $E$ is non-decreasing and concave along affine curves. Additionally, the following estimates hold: $
	\int_X (u-v) H_m(u)  \leq E(u) -E(v) \leq \int_X (u-v)H_m(v).$\\
\noindent (iii) If $v\leq u$ then, $
\frac{1}{m+1} \int_X (u-v) H_m(v) \leq  E(u) - E(v) \leq  \int_X (u-v) H_m(v). $ In particular, $E(v) \leq E(u)$. 
\end{prop}

One can thus extend $E$ to $\SH_m(X,\omega)$ by 
$$
E(u) := \inf \{E(v) \setdef v\in \SH_m(X,\omega)\cap L^{\infty}, \ v \geq u\}. 
$$
A function $u\in \SH_m(X,\omega)$ belongs to $\Eone$ iff $E(u)>-\infty$. 

Following \cite{Dar17AJM,Dar15} we introduce the functional $I_1$
$$
I_1(u,v) := \int_X |u-v| \left ( H_m(u)+H_m(v) \right ). 
$$
\begin{prop}\label{prop: continuity of I1}
	Assume that $u_j \in \Eone$ is a  monotone sequence converging to $u \in \Eone$. Then $I_1(u_j,u) \to 0$ and $E(u_j) \to E(u)$.  
\end{prop}
\begin{proof}
	The proof is an obvious modification of the Monge-Amp\`ere case, see e.g.  \cite{BEGZ10}, \cite[Proposition 2.7]{DDL3}. 
\end{proof}

\section{Relative Potential Theory} \label{sect: Relative PP}

\subsection{Monotonicity of the complex Hessian mass}\label{sect: monotonicity}
In this section we extend the monotonicity results of \cite{WN19}, \cite{DDL2} to the Hessian cases $m<n$. The proof is new in the Monge-Amp\`ere case. 

Recall that we normalize $\omega$ such that $\int_X \omega^n=1$.  We first establish the following slope formula:

\begin{lemma}
\label{lem: slope formula}
For any $u\in \SH_m(X,\omega)$ we have 
$$
\lim_{s\to +\infty} \frac{E(\max(u,-s))}{s} = -1 + \frac{1}{m+1} \sum_{k=0}^m \int_X H_k(u). 
$$
\end{lemma}
\begin{proof}
We set $u^s := \max(u,-s)$ and compute 
\begin{eqnarray*}
\frac{(m+1)E(u^s)}{s} =  \sum_{k=0}^m \int_{\{u>-s\}} \frac{u}{s}H_k(u^s)  -  \sum_{k=0}^m \int_{\{u\leq -s\}}  H_k(u^s). 
\end{eqnarray*}
We note that, by the Lemma \ref{lem: Identity principle}, ${\bf 1}_{\{u>-s\}}H_k(u^s) = {\bf 1}_{\{u>-s\}} H_k(u)$. Thus we can continue the above computation to write 
$$
\frac{(m+1)E(u^s)}{s} =  \sum_{k=0}^m \int_{\{u>-s\}} \frac{u}{s}H_k(u)  -  \sum_{k=0}^m \int_{\{u\leq -s\}}  H_k(u^s). 
$$ 
The functions ${\bf 1}_{\{u>-s\}} \frac{u}{s}$ are uniformly bounded and converge to $0$ outside the $m$-polar set $\{u=-\infty\}$.  Since $H_k(u)$ does not charge $m$-polar sets, we see that 
$$
\lim_{s\to +\infty}\sum_{k=0}^m \int_{\{u>-s\}} \frac{u}{s}H_k(u)  =0. 
$$
On the other hand, by Lemma \ref{lem: Identity principle} again we have 
\begin{eqnarray*}
1= \int_{X} H_k(u^s) &=&  \int_{\{u>-s\}} H_k(u^s) + \int_{\{u\leq -s\}} H_k(u^s)
\\ &=& \int_{\{u>-s\}}  H_k(u) + \int_{\{u\leq -s\}} H_k(u^s).  
\end{eqnarray*}
Letting $s\to +\infty$ we obtain the result. 
\end{proof}
\begin{prop}
	\label{prop: mass equality}
	Let $u,v \in \SH_m(X,\omega)$, and ssume that there exists a constant $C\in \mathbb{R}$ such that $v-C \leq u\leq v+C$ on $X$. Then 
	$$
	\int_X H_k(u) = \int_X H_k(v), \ \forall k\in \{0,...,m\}. 
	$$
\end{prop}
\begin{proof}
Fix $1\leq l\leq m$, and observe that $\SH_m(X,\omega) \subset \SH_l(X,\omega)$.  For each $s>0$ set $u^s := \max(u,-s)$.  By assumption we have 
$$
v^s -C \leq u^s \leq v^s +C. 
$$
Hence, the monotonicity of the  energy $E_l$ \cite[Lemma 6.3]{LN15} gives, for all $s>0$,
$$
\frac{E_l(v^s) -C}{s} \leq \frac{E_l(u^s)}{s} \leq \frac{E_l(v^s) + C}{s}. 
$$
Letting $s\to +\infty$ and using Lemma \ref{lem: slope formula} we obtain the following equalities for $l=1,...,m$, which imply the result: 
$$
\sum_{k=0}^l \int_X H_k(u) = \sum_{k=0}^l  \int_X H_k(v). 
$$
\end{proof}

\begin{theorem}\label{thm: lsc of Hes measure}
	Assume that $u_1^j,...,u_m^j$ are sequences of $\omega$-$m$-sh functions converging in $m$-capacity to  $\omega$-$m$-sh functions $u_1,...,u_m$. Let $\chi_j$ be a sequence of positive uniformly bounded quasi-continuous functions which  converges in capacity to $\chi$.  Then, 
	$$
	\liminf_{j\to +\infty}\int_X \chi_j H_m(u_1^j,...,u_m^j) \geq \int_X \chi H_m(u_1,...,u_m). 
	$$
	In particular, if $\Omega\subset X$ is a quasi-open set then 
	$$
	\liminf_{j\to +\infty} \int_{\Omega} H_m(u_1^j,...,u_m^j) \geq \int_{\Omega} H_m(u_1,...,u_m). 
	$$
\end{theorem}

\begin{proof}
	We borrow the ideas of \cite{DDL2}. Fix $C>0$, $\varepsilon>0$, and set 
	$$
	U_C^j := \cap_{p=1}^m \{u_p^j >-C\},  \; f_{C,\varepsilon}^j :=  \prod_{p=1}^m\frac{\max(u_p^j+C,0)}{\max(u_p^j+C,0)+\varepsilon}. 
	$$
	Observe that $0\leq f_{C,\varepsilon}^j\leq 1$ and  $f_{C,\varepsilon}^j$ vanishes outside $U_C^j$. We thus have 
	\begin{flalign*}
		\liminf_{j\to +\infty}& \int_X  \chi_j H_m(u_1^j,...,u_m^j)  \geq  \liminf_{j\to +\infty}\int_{U_C^j} \chi_j H_m(u_1^j,...,u_m^j)  \\
		&=  \liminf_{j\to +\infty}\int_{U_C^j} \chi_j H_m(\max(u_1^j,-C),...,\max(u_m^j,-C))  \\
		&\geq  \liminf_{j\to +\infty}\int_X \chi_j f_{C,\varepsilon}^j H_m(\max(u_1^j,-C),...,\max(u_m^j,-C)),
	\end{flalign*}
	where in the second line we have used the plurifine locality. For fixed $C>0$ the functions $\max(u_p^j,-C)$ are uniformly bounded, hence we can use \cite[Proposition 3.12]{Lu13}, which is a direct adaptation of the case $m=n$,  to continue the above inequality in the following way
	\begin{flalign*}
		\liminf_{j\to +\infty} \int_X \chi_j H_m(u_1^j,...,u_m^j) 
		& \geq \int_X \chi f_{C,\varepsilon} H_m(\max(u_1-C),...,\max(u_m-C))\\
		&\geq \int_{U_C} \chi f_{C,\varepsilon} H_m(\max(u_1-C),...,\max(u_m-C))\\
		&\geq \int_{U_C} \chi f_{C,\varepsilon} H_m(u_1,...,u_m).
	\end{flalign*}
	In the last line above we have used Lemma \ref{lem: Identity principle}. We now let $\varepsilon\to 0$ and then $C\to +\infty$ to conclude the proof of the first statement. 
	
	To prove the last statement we follow the lines above with $\chi_j = 1$, $X$ replaced by $\Omega$, and we use Theorem \ref{thm: convergence quasi set}. 
\end{proof}

As shown in Theorem \ref{thm: lsc of Hes measure}, the (non-$m$-polar) Hessian measure is lower semicontinuous along sequences converging in $m$-capacity.   We give below  sufficient conditions for the convergence.

\begin{coro}
\label{coro: increasing convergence}
Assume that $u_1^j,...,u_m^j$ are  sequences of $\omega$-$m$-sh functions which increase a.e. to $\omega$-$m$-sh functions $u_1,...,u_m$. Then 
$$
H_m(u_1^j,...,u_m^j) \to H_m(u_1,...,u_m)
$$
weakly in the sense of measures. 
\end{coro}
\begin{proof}
It is a direct consequence of Theorem \ref{thm: monotonicity} and Theorem \ref{thm: lsc of Hes measure}.
\end{proof}

\begin{lemma}
\label{lem: GLZ cap}
Let $\mu$ be a positive measure vanishing on $m$-polar sets.  Then there exists a continuous function $f : [0,+\infty) \rightarrow [0,+\infty)$ such that, for all Borel set $E$,  
$$
\mu(E) \leq f(\capa_m(E)).
$$
\end{lemma}
\begin{proof}
The proof is an easy adaptation of \cite{GLZJDG}. We repeat this argument here for the reader's convenience. It follows from \cite[Theorem 1.3]{LN15} that there exists $\psi \in \Em$ such that $\sup_X \psi=0$ and $\mu = CH_m(\psi)$, for some positive constant $C$.  

Let $E\subset X$ be a Borel set such that $\capa_m(E)>0$. For $t>1$ we have 
\begin{flalign*}
\mu(E\cap \{\psi >-t\}) &=  C\int_E H_m(\max(\psi,-t)) \leq C t^m \capa_m(E). 
\end{flalign*}
Let $\chi : (-\infty,0) \rightarrow (-\infty,0)$ be a convex increasing function such that $\chi(-\infty)= -\infty$ and $C_1:= \int_X |\chi(\psi)| d\mu <+\infty$. For $t>1$ we have 
$$
\mu(\psi\leq -t) \leq \frac{1}{|\chi(-t)|}\int_X |\chi(\psi)| d\mu = \frac{C_1}{|\chi(-t)|}. 
$$
Choosing $t$ such that $t^{m+1}= \max(\capa_m(E)^{-1}, 1)$, we finish the proof of the lemma. 
\end{proof}
\begin{theorem}
\label{thm: dominated convergence}
Assume that $u_j \in \SH_m(X,\omega)$ decreases to $u\in \SH_m(X,\omega)$. If there exists a non-$m$-polar positive measure $\mu$  such that 
$$
H_m(u_j) \leq \mu, \forall j,
$$
 then $H_m(u_j)$ weakly converges to $H_m(u)$. 
\end{theorem}
\begin{proof}
By Theorem \ref{thm: lsc of Hes measure} we have that $H_m(u) \leq \mu$ and it remains to prove the convergence of the total mass.


We can assume that $\sup_X u_j=\sup_X u= 0$. For a function $v$ and a constant $t$ we set $v^t := \max(v,-t)$.  For all $t>0$ we have 
$$
\mu(u \leq -t)  \leq f\left( \capa_m(u \leq -t) \right ),
$$
where $f$ is the continuous function in Lemma \ref{lem: GLZ cap}. By continuity of $f$ we have 
$$
\lim_{t\to+\infty}f\left( \capa_m(u \leq -t) \right )= 0.
$$ 
Therefore,  fixing $\varepsilon>0$, for $t>0$ large enough we  have
$$
\int_{\{u\leq -t\}} H_m(u_j) \leq \mu(u \leq -t)  \leq f\left( \capa_m(u \leq -t) \right ) \leq  \varepsilon,  \ \forall j.
$$
Thus, for fixed $s>t$ we have
\begin{eqnarray*}
\int_X H_m(u_j)  \leq    \int_{\{u\geq -t\}}  H_m(u_j) +    \varepsilon \leq  \int_{\{u\geq -t\}} H_m(u_j^s) +\varepsilon.
\end{eqnarray*} 
Here, we use Lemma \ref{lem: Identity principle} and the assumption that $u_j \geq u$ to have that 
$$
\Id_{\{u>-s\}} H_m(u_j) =\Id_{\{u>-s\}}H_m(u_j^s),
$$ 
hence
\begin{flalign*}
 \int_{\{u\geq -t\}}  H_m(u_j)  &= \int_{\{u\geq -t\}}  \Id_{\{u>-s\}} H_m(u_j)= \int_{\{u\geq -t\}}  \Id_{\{u>-s\}} H_m(u_j^s)\\ &=\int_{\{u\geq -t\}} H_m(u_j^s). 
\end{flalign*}
Since $\{u\geq -t\}$ is quasi compact and $u_j^s$ are uniformly bounded, letting $j\to +\infty$ we obtain 
$$
\limsup_{j}\int_X H_m(u_j) \leq \int_{\{u\geq -t\}}  H_m(u^s) + \varepsilon =\int_{\{u\geq -t\}}  H_m(u)+ \varepsilon.
$$ 
Letting $t\to +\infty$, and then $\varepsilon\to 0$ we arrive at the conclusion. 
\end{proof}

We are now in the position to prove the main result of this section. 

\begin{theorem}
	\label{thm: monotonicity}
	Let $u_1,...,u_m,v_1,...,v_m \in \SH_m(X,\omega)$ and assume that $u_j$ is more singular than $v_j$ for all $j$. Then  
	$$
	\int_X H_m(u_1,...,u_m)   \leq \int_X H_m(v_1,...,v_m).
	$$
\end{theorem} 

\begin{proof}
	We first assume that  $u_p$ has the same singularity as $v_p$ for all $p=1,...,m$. For $t=(t_1,...,t_m) \in [0,1]^m$  with $\sum_{p=1}^m t_p =1$, we set 
	$$
	\phi_t := \sum_{p=1}^m t_p u_p, \quad \psi_t :=  \sum_{p=1}^m t_p v_p.
	$$
	Then $\phi_t,\psi_t\in \SH_m(X,\omega)$ have the same singularity. It thus follows from Proposition \ref{prop: mass equality} that 
	$$
	\int_X H_m(\phi_t) = \int_X H_m(\psi_t).
	$$
	From this and Lemma \ref{lem: multilinearity} we obtain an equality between two polynomials in $(t_1,...,t_m)$. Identifying the coefficients we obtain 
	$$
	\int_X H_m(u_1,...,u_m) =\int_X H_m(v_1,...,v_m). 
	$$
	To treat the general case we define, for $C>0$, $w_p^C := \max(u_p,v_p-C)$. Then the previous step yields 
	$$
	\int_X H_m(w_1^C,...,w_m^C) =\int_X H_m(v_1,...,v_m). 
	$$
	Letting $C\to +\infty$ and using Theorem \ref{thm: lsc of Hes measure} we arrive at the conclusion. 
\end{proof}

Having the monotonicity theorem in hand most of the pluripotential tools in \cite{DDL2,DDL4} can be adapted directly to the Hessian setting. Since the references \cite{DDL2,DDL4} are quite recent, we give the full details.

\subsection{Envelopes} \label{sect: envelope} Let $f$ be a function on $X$. We define
$$
P_{(\omega,m)}(f) := \left ( \sup \{ u  \setdef   u\in \SH_m(X,\omega), \ u \leq f \}\right )^*,
$$
where the $*$ operator means the upper semicontinuous regularization. 
Following \cite{RWN14}, \cite{DDL2,DDL4} we define 
$$
P_{(\omega,m)}[f] := \left (\lim_{C\to +\infty} P_{(\omega,m)} (\min(f+C,0))\right )^*. 
$$
 If $(\omega,m)$ is fixed we will simply denote these envelopes by $P(f)$ and $P[f]$. For $u_1,...,u_N \in \SH_m(X,\omega)$ we denote $P(u_1,...,u_N) := P(\min(u_1,...,u_N))$.

\begin{lemma}\label{lem: mass u Pu}
If $u_1,...,u_m, v_1,...,v_m \in \SHm$ satisfy $P[u_p]=P[v_p]$, for all $p$, then 
$$
\int_X H_m(u_1,...,u_m) =\int_X H_m(v_1,...,v_m).
$$ 
\end{lemma}
\begin{proof}
For each $C>0$ $P(u_j+C,0)$ has the same singularity as $u_j$, hence by Theorem \ref{thm: monotonicity},  
$$
\int_X H_m(u_1,...,u_m) = \int_X H_m(P(u_1+C,0),...,P(u_m+C,0)). 
$$
Letting $C\to +\infty$,  Corollary \ref{coro: increasing convergence} ensures that
$$
\int_X H_m(u_1,...,u_m) = \int_X H_m(P[u_1],...,P[u_m]).
$$
The same arguments apply for $v_1,...,v_m$, yielding the result. 
\end{proof}

\begin{lemma}\label{lem: concavity of P}
	If $u,v \in \SH_m(X,\omega)$ and $t \in (0,1)$ then 
	$$
	P[t u+ (1-t) v] \geq t P[u] + (1-t) P[v].
	$$
\end{lemma}
\begin{proof}
	For each $C>0$ we have that $t P(u+C,0) + (1-t) P(v+C,0)$ is $\omega$-$m$-sh and it is smaller than $\min(tu+(1-t)v+C,0)$. Thus
	$$
	P(tu+(1-t)v +C,0)  \geq t P(u+C,0) + (1-t) P(v+C,0),
	$$
	hence letting $C\to +\infty$ we obtain the result. 
\end{proof}

\begin{prop}
	\label{prop: orthogonal}
	Assume that $f = a\varphi-b\psi$, where $\varphi,\psi\in \SH_m(X,\omega)$, and $a, b$ are positive constants.   If $P(f)\not \equiv -\infty$ then 
	$$
	\int_{\{P(f)<f\}} H_m(P(f)) =0. 
	$$
\end{prop}
Here,  the function $f= a\varphi- b\psi$ is well-defined in the complement of a pluripolar set and the inequality $u \leq  a\varphi-b\psi$, for $u \in \SHm$, means  $u+ b\psi \leq a \varphi$ on $X$.
\begin{proof}
We first assume that $\varphi$ is continuous. Then $P(f)$ is bounded. Let $\psi_j$ be a sequence of continuous $\omega$-$m$-sh functions decreasing to $\psi$ and set $f_j := a\varphi-b\psi_j$, $u_j=P(f_j)$. By \cite{LN15} we have 
	$$
	\int_X \min(f_j-u_j,1) H_m(u_j) =0, \ \forall j. 
	$$
Let $u := (\lim_{j\to +\infty} u_j)^*$. It follows from  \cite[Proposition 3.12]{Lu13} that
	$$
	\int_X (\min(f-u,1) H_m(u) =0,
	$$
	hence $\int_{\{u<P(f)\}} H_m(u) =0$ and the domination principle \cite[Lemma 3.5]{DL15} gives $u=P(f)$. By the above equality we also have that $H_m(u)$ vanishes in $\{u<f\}$. 

	We now treat the general case. Let $\varphi_j$ be a sequence of continuous $\omega$-$m$-sh functions decreasing to $\varphi$ and set $f_j := a \varphi_j -b\psi$. Then $P(f_j) \searrow P(f)$.  From the first step we have
	$$
	\int_X \min(f_j-P(f_j),1) H_m(P(f_j)) =0, \ \forall j. 
	$$
	Letting $j\to +\infty$ and using Theorem \ref{thm: lsc of Hes measure} we arrive at the conclusion. 
\end{proof}

From Proposition \ref{prop: orthogonal} and Proposition \ref{prop: viscosity} we obtain the following : 
\begin{coro}\label{cor: Hes rooftop}
Let $u,v\in \SH_m(X,\omega)$ be such that $P(u,v)\in \SH_m(X,\omega)$. Then 
$$
H_m(P(u,v)) \leq {\bf 1}_{\{P(u,v) =u\}} H_m(u) + {\bf 1}_{\{P(u,v)= v\}} H_m(v). 
$$ 
In particular, $H_m(P[u]) \leq {\bf 1}_{\{P[u]=0\}} \omega^n$. Finally, if $H_m(u) \leq \mu$ and $H_m(v) \leq \mu$, for a non-$m$-polar positive measure $\mu$, then $H_m(P[u,v]) \leq \mu$. 
\end{coro}

\begin{definition}
A function $\phi \in \SH_m(X,\omega)$ is a model potential if $\int_X H_m(\phi)>0$ and $P[\phi]=\phi$. 

Given a model potential $\phi$, the class $\Ephi:= \mathscr{E}_{\phi}(X,\omega,m)$ consists of functions $u\in \SH_m(X,\omega)$ such that $u$ is more singular than $\phi$ and $\int_X H_m(u) = \int_X H_m(\phi)$.  
\end{definition}

\subsection{Comparison principle}
\begin{theorem}\label{thm: CP}
Let $\phi_2,...,\phi_{m}, u,v \in \SH_m(X,\omega)$ and assume that $P[u]\geq P[v]$.  Then 
$$
\int_{\{u<v\}} H_m(v,\phi_2,...,\phi_{m}) \leq \int_{\{u<v\}} H_m(u,\phi_2,...,\phi_{m}). 
$$
\end{theorem}
\begin{proof}
Fix $\varepsilon>0$ and set $v_{\varepsilon}:= \max(v-\varepsilon,u)$. Then $P[v^{\varepsilon}] =P[u]$, hence by Lemma \ref{lem: mass u Pu} we have 
$$
\int_{X} H_m(v^{\varepsilon},\phi_2,...,\phi_m) = \int_X H_m(u,\phi_2,...,\phi_m). 
$$ 
By Lemma \ref{lem: Identity principle} we also have 
$$
\int_X H_m(v^{\varepsilon},\phi_2,...,\phi_m)  \geq  \int_{\{u>v-\varepsilon\}} H_m(u,\phi_2,...,\phi_m) + \int_{\{u<v-\varepsilon\}} H_m(v,\phi_2,...,\phi_m). 
$$
Comparing these we arrive at 
$$
\int_{\{u<v-\varepsilon\}} H_m(v,\phi_2,...,\phi_m) \leq \int_{\{u\leq v-{\varepsilon}\}} H_m(u,\phi_2,...,\phi_m). 
$$
Letting $\varepsilon \to 0^+$ we obtain the result. 
\end{proof}

\subsection{Domination principle}

\begin{lemma}\label{lem: non-collapsing}
Assume that $u\in \SH_m(X,\omega)$ and $\int_X H_m(u) >0$. If $E\subset X$ is a Borel set such that $\int_E \omega^n>0$ then there exists $v \in \SH_m(X,\omega)$ such that $v$ has the same singularity as $u$ and 
$$
\int_E H_m(v) >0.
$$
\end{lemma}

\begin{proof} 

Let $\phi \in \SH_m(X,\omega) \cap L^{\infty}(X)$ be such that $H_m(\phi) = c{\bf 1}_E \omega^n$, where $c>0$ is a normalization constant. For $t>0$ set $u_t := P(\min(u+t,\phi))$. Corollary \ref{cor: Hes rooftop} gives 
$$
\int_{X\setminus E}H_m(u_t) \leq \int_{X\setminus E}  {\bf 1}_{\{u_t = u+t\}}H_m(u) \leq \int_{\{u \leq \phi-t\}} H_m(u). 
$$
Thus, for $t>0$ large enough we have $\int_{X\setminus E} H_m(u_t) < \int_X H_m(u)=\int_X H_m(u_t)$, where the last equality follows from Theorem \ref{thm: monotonicity} since $u_t$ has the same singularity as $u$. For such $t$ we thus have  $\int_E H_m(u_t) >0$, finishing the proof. 
\end{proof}

\begin{theorem}\label{thm: domination principle}
Assume that $u,v\in \SH_m(X,\omega)$ and $u$ is less singular than $v$. If $\int_{\{u<v\}}H_m(u)=0$ and $\int_X H_m(u)>0$ then $u\geq v$. 
\end{theorem}

\begin{proof} 
Assume by contradiction that $E:= \{u<v\}$ is not empty. Then $\int_E \omega^n>0$ and hence Lemma \ref{lem: non-collapsing} provides us with $h \in \SH_m(X,\omega)$ having the same singularity as $u$ such that $\int_E H_m(h)>0$. We can assume that $h\leq u$. For $t\in (0,1)$ set $v_t := th +(1-t)v$.  Then $E_t := \{u<v_t\} \subset E$ and $\cup E_t =E$. Hence for $t$ small enough we have $\int_{E_t} H_m(h)>0$. But the comparison principle gives
$$
t^m \int_{E_t} H_m(h) \leq \int_{E_t}H_m(v_t) \leq \int_{E_t} H_m(u) =0,
$$ 
which is a contradiction.
\end{proof}

\begin{coro}
\label{cor: Darvas criterion}
If $\phi$ is a model potential then $u\in \Ephi$ iff $P[u]=\phi$. 
\end{coro}

\begin{proof} 
If $u \in \Ephi$ then the domination principle, Theorem \ref{thm: domination principle},  gives $P[u]=\phi$. Assume now that $P[u]=\phi$. Since $P[u]$ is the increasing limit of $P(\min(u+t,0))$ as $t\to +\infty$,  Theorem \ref{thm: lsc of Hes measure} gives $\int_X H_m(u) =\int_X H_m(P[u])$, hence $u\in \Ephi$. 
\end{proof}

\begin{coro}\label{coro: maximal}
If $\phi$ is a model potential and $u\in \Ephi$ then $u-\sup_X u \leq \phi$.
\end{coro}

\begin{lemma}
	 If  $u,v\in \SH_m(X,\omega)$ and $P(u,v) \in \SH_m(X,\omega)$ then $P[\min(u,v)] = P[P(u,v)]$. 
\end{lemma}

\begin{proof}
	By definition we have 
	\begin{flalign*}
		P[\min(u,v)] & = \left(\lim_{C\to +\infty} P(\min(u+C,v+C,0))\right)^*\\
		 &\leq \left(\lim_{C\to +\infty} P(\min(P(u,v) + C,0)\right)^*  = P[P(u,v)].
	\end{flalign*}
	The reverse inequality follows directly from the definition. 
\end{proof}

\subsection{Strongly $m$-positive currents}
We borrow the idea in \cite{DDL5}. 
\begin{theorem}
\label{thm: strict m positive}
Assume that $b>1$, $u,v \in \SH_m(X,\omega)$, $u\leq v$,  and
$$
\int_X H_m(v)  > b^m \left (\int_X H_m(v)-\int_X H_m(u) \right ).
$$
Then $P(bu-(b-1) v) \in \SH_m(X,\omega)$. 
\end{theorem}
If $v=0$ and $\int_X H_m(u)>0$ then by the above result there exists $b>1$ such that $P(bu) \in \SHm$. Therefore $b^{-1}P(bu)$ is a strongly $\omega$-$m$-sh function lying below $u$. This will be used in proving the existence of solutions to complex Hessian equations with prescribed singularity.

\begin{proof}
We can assume that $P[v]=v$. 

For $t>0$ set $u_t := \max(u,v-t)$, $\varphi_t := P(bu_t-(b-1)v)\in \SH_m(X,\omega)$, and $D:=\{ \varphi_t = bu_t -(b-1)v\}$. Then  $b^{-1}\varphi_t +(1-b^{-1})v \leq u_t$ with equality on $D$, hence Proposition \ref{prop: viscosity} gives 
\begin{flalign*}
{\bf 1}_D b^{-m}H_m(\varphi_t) \leq  {\bf 1}_D H_m(b^{-1}\varphi_t+(1-b^{-1})v) \leq \Id_D H_m(u_t). 
\end{flalign*}
Fix $s<t$. By the above inequality and Proposition \ref{prop: orthogonal} we have 
\begin{eqnarray*}
\int_{\{\varphi_t \leq v -bs\}} H_m(\varphi_t) &\leq &  b^m \int_{\{bu_t \leq bv -bs\}} H_m(u_t) = b^m\int_{\{u \leq v -s\}} H_m(u_t)\\
& =& b^m \left (\int_X H_m(v)- \int_{\{u>v-s\}} H_m(u_t)\right)\\
&=& b^m\left (\int_X H_m(v)-\int_{\{u>v-s\}} H_m(u)\right), 
\end{eqnarray*}
where in the last line we use Lemma \ref{lem: Identity principle}. 

We want to prove that $\varphi_t$ decreases to some $\omega$-$m$-subharmonic function on $X$. Assume by contradiction that it is not the case. Then $\sup_X \varphi_t$ decreases to $-\infty$. Since $v=P[v]$, by Corollary \ref{coro: maximal} we have $\varphi_t \leq v + \sup_X \varphi_t$. Thus, for $s>0$ fixed  and for $t$ large enough $\{\varphi_t \leq  v -s\} =X$.  Fixing  $s>0$ and letting $t\to +\infty$ we obtain
$$
\int_X H_m(v)  \leq b^m \left(\int_X H_m(v)- \int_{\{u>-s\}} H_m(u)\right).  
$$
Now, letting $s\to +\infty$ we obtain a contradiction with the assumption. 
\end{proof}

\begin{coro}
\label{coro: subextension}
Assume that $u,v \in \SH_m(X,\omega)$, $P[u] = P[v]$ and $\int_X H_m(v)>0$.  Then for all $b>1$, 
$P(bu-(b-1)v) \in \mathscr{E}_{P[v]}$. 
\end{coro}

\begin{proof}
We can assume that $u,v \leq 0$. Then $u \leq P[u] =P[v]$. Fix $b>1$.  We first observe that $P(bu -(b-1)P[v]) \in \SH_m(X,\omega)$ as follows from Theorem \ref{thm: strict m positive}. Hence $P(bu-(b-1)v) \in \SH_m(X,\omega)$.  For $t>b$ we have 
$$
u\geq P(bu-(b-1)P[v]) \geq bt^{-1} P(tu-(t-1)P[v]) + (1-bt^{-1})P[v].
$$
By monotonicity of mass, see Theorem \ref{thm: monotonicity}, we have 
$$
\int_X H_m(P(bu-(b-1)P[v]) ) \geq (1-bt^{-1})^m \int_X H_m(P[v]). 
$$
Letting $t\to +\infty$ we obtain $P(bu-(b-1)P[v]) \in \mathscr{E}_{P[v]}$.  
We also have
$$
b^{-1} P(bu -(b-1)v)  + (1-b^{-1}) v \leq u,
$$
hence, by Lemma \ref{lem: concavity of P} we have 
$
b^{-1} P[P(bu -(b-1)v)] + (1-b^{-1})P[v] \leq P[u] =P[v],
$
which implies $P[P(bu-(b-1)v)] \leq P[v]$. But we have already proved that 
$$
P[P(bu-(b-1)v)] \geq P[P(bu-(b-1)P[v])] =P[v].
$$
 We thus have equality.   
\end{proof}

\begin{coro}
\label{coro: subextension 2}
Assume that $u,v \in \SH_m(X,\omega)$ are such that 
$P[u]\geq P[v]$ and $\int_X H_m(v)>0$. 
Then, for all $b>1$, $P(bu-bv) \in \Em$. 
\end{coro}

\begin{proof}
We can assume that $u,v\leq 0$. Then $v\leq P[v] \leq P[u]$, hence $u \leq \max(u,v) \leq P[u]$. It thus follows that $\max(u,v)\in \mathscr{E}_{P[u]}$.  Hence by Corollary \ref{coro: subextension} we have, for all $b>1$, $P(bu-bv) \geq P(bu -(b-1)\max(u,v)) \in \SH_m(X,\omega)$. For $t>b>1$, we have 
$$
P(bu-bv) \geq bt^{-1}P(tu -(t-1) v) + (1-bt^{-1})v.
$$
Comparing total mass and letting $t\to +\infty$ we obtain the result. 
\end{proof}

\begin{prop}
	\label{prop: rooftop envelope 1}
	If $\phi$ is a model potential and $u,v\in \Ephi$ then $P(u,v) \in \Ephi$. 
\end{prop}

\begin{proof}
	The proof is similar to that of Theorem \ref{thm: strict m positive}. We first prove that $P(u,v) \in \SH_m(X,\omega)$. 
	For $t>0$ set $u_t := \max(u,\phi-t)$, $v_t := \max(v,\phi-t)$,  and $\varphi_t := P(u_t,v_t) )\in \Ephi$. 
	We want to prove that $\varphi_t$ decreases to some $\omega$-$m$-subharmonic function on $X$. Assume by contradiction that it is not the case. Then $\sup_X \varphi_t$ decreases to $-\infty$. Since $\phi=P[\phi]$, by Corollary \ref{coro: maximal} we have $\varphi_t \leq \phi + \sup_X \varphi_t$. Thus, for $s>0$ fixed  and for $t$ large enough we have $\{\varphi_t \leq  \phi -s\} =X$. Using this and Corollary  \ref{cor: Hes rooftop} we obtain
	\begin{flalign*}
		\int_X H_m(\phi) &= \int_{\{\varphi_t \leq \phi-s\}} H_m(\varphi_t) \leq \int_{\{u\leq \phi-s\}} H_m(u_t) +  \int_{\{v\leq \phi-s\}} H_m(v_t)\\
		&= 2 \int_X H_m(\phi) - \int_{\{u>\phi-s\}} H_m(u) - \int_{\{v>\phi-s\}} H_m(v).
	\end{flalign*}
	Letting $s\to +\infty$ we obtain $\int_X H_m(\phi) \leq 0$,  a contradiction.  Thus $P(u,v) \in \SH_m(X,\omega)$. 
	
	Now, by Corollary \ref{coro: subextension} we have that, for all $b>1$, $u_b := P(b u -(b-1)\phi) \in \Ephi$ and $v_b:= P(b v -(b-1) \phi) \in \Ephi$. Hence by the previous step we have $P(u_b,v_b) \in \SH_m(X,\omega)$. We also have that $P(u,v)$ is more singular than $\phi$ and 
	$$
	P(u,v) \geq b^{-1}P(u_b,v_b) +(1-b^{-1})\phi.
	$$  
	Thus $\int_X H_m(P(u,v)) \geq (1-b^{-1})^m \int_X H_m(\phi)$. Letting $b\to +\infty$ we arrive at the conclusion. 
	\end{proof}

\section{A metric on $\Eone$} \label{sect: metric}
Following \cite{DDL3}, we introduce a metric on $\mathscr{E}^1(X,\omega,m)$ and use it to construct subextensions of a family of $\omega$-$m$-subharmonic functions. Most of this section are taken from \cite{DDL3} but we recall them for completeness, since we  will crucially use Theorem \ref{thm: compactness}.
\subsection{Define a metric on $\Eone$}

Given $u,v \in \Eone$ we define 
$$
d(u,v) := E(u)+E(v) -2 E(P(u,v)). 
$$ 
Here $P(u,v) := P(\min(u,v))$ is the largest $\omega$-$m$-sh function lying below $\min(u,v)$. This is called the rooftop envelope \cite{DR16} which plays a crucial role in the recent developments in Geometric Pluripotential Theory (see \cite{Dar18S}).  The proof of \cite[Theorem 3.6]{Dar17AJM}, applied to the Hessian setting, shows that $P(u,v) \in \Eone$.  Arguing as in \cite{DDL3} we can show that $d$ is a metric and $(\Eone,d)$ is compete, along with many useful properties. 
\begin{lemma}\label{lem: basic properties} Let $u,v\in \Eone$. Then the following hold:\\
	(i) If $u\leq v$ then $d(u,v) =E(v)- E(u)$.\\
	(ii) If $u\leq v\leq w$ then $d(u,v)+d(v,w)= d(u,w)$. \\	
	(iii) (Pythagorean formula) $d(u,v) =d(u,P(u,v))+d(v,P(u,v))$. 
\end{lemma}

\begin{prop}
	\label{prop: derivative}
	Let $u,v$ be bounded $\omega$-$m$-sh functions, and set 
	$$
	\varphi_t:= P((1-t)u+tv, v), \ t\in [0,1].
	$$ 
	Then 
	$$
	\frac{d}{dt} E(\varphi_t) = \int_X (v-\min(u,v)) H_m(\varphi_t), \ \forall t\in [0,1]. 
	$$
\end{prop}
\begin{proof}
	We will only prove the formula for the right derivative as the same argument can be applied to treat the left derivative. Fix $t\in [0,1)$ and let $s>0$ be small. For notational convenience we set 
	$$
	f_t(x):= \min((1-t)u(x)+tv(x),v(x)), \ x \in X, \ t\in [0,1].
	$$ 
It follows from \cite[Theorem 3.2]{LN15} that $H_m(\varphi_t)$ is supported on the set $\{\varphi_t=f_t\}$. Combining this  with the concavity of the  energy $E$, see Proposition \ref{prop: basic fact on energy},  we obtain  
	\begin{flalign*}
		 E(\varphi_{t+s}) - E(\varphi_t) &\leq  \int_X (\varphi_{t+s}-\varphi_t) H_m(\varphi_t) \\
		 &= \int_X (\varphi_{t+s}-f_t) H_m(\varphi_t) \leq   \int_X (f_{t+s}-f_t) H_m(\varphi_t). 
	\end{flalign*}
	On the other hand we have that $f_{t+s}-f_t = s (v-\min(u,v))$. It thus follows that 
	$$
	\lim_{s\to 0^+} \frac{ E(\varphi_{t+s})- E(\varphi_t)}{s} \leq  \int_X (v-\min(u,v)) H_m(\varphi_t). 
	$$
	We use the same argument to prove the reverse inequality: 
	\begin{flalign*}
		E(\varphi_{t+s}) - E(\varphi_t) & \geq  \int_X (\varphi_{t+s}-\varphi_t) H_m(\varphi_{t+s}) = \int_X (f_{t+s}-\varphi_t) H_m(\varphi_{t+s})\\
	& \geq   \int_X (f_{t+s}-f_t) H_m(\varphi_{t+s}) = s  \int_X (v-\min(u,v)) H_m(\varphi_{t+s}). 
	\end{flalign*}
	As $s\to 0^+$ we have that $\varphi_{t+s}$ converges uniformly to $\varphi_t$. Moreover, $v-\min(u,v)$ is a bounded quasi continuous function on $X$, hence \cite[Proposition 3.12]{Lu13} gives 
	$$
	\lim_{s\to 0^+} \frac{E(\varphi_{t+s})- E(\varphi_t)}{s} \geq  \int_X (v-\min(u,v))H_m(\varphi_{t}). 
	$$
	This completes the proof. 
\end{proof}

\begin{coro}
	\label{cor: derivative}
	Let $u,v,\varphi_t$ be as in Proposition \ref{prop: derivative}. Then 
	\[
	E(v)-E(P(u,v)) = \int_0^1 \int_X (v-\min(u,v)) H_m(\varphi_t) dt. 
	\]
\end{coro}
\begin{prop}
	\label{prop:  inequality max P}
	If  $u,v\in \Eone$ then $d(\max(u,v),u)\geq d(v,P(u,v))$. 
\end{prop}
\begin{proof}
Set $\varphi=\max(u,v)$, $\psi= P(u,v)$. Observe that since $v\geq \psi$ and $\varphi \geq u$, the inequality to be proved is equivalent to $E(v)-E(\psi)\leq E(\varphi)-E(u)$. 

Recall  that for any $w \in \Eone$ the sequence of bounded potentials  $w_k:= \max(w, -k)$ decreases to $w$. 
	Consequently, using approximation, we can assume that both $u$ and $v$ (hence also $\varphi$ and $\psi$) are bounded. Using the formula for the derivative of  $t\mapsto E((1-t)u + t\varphi)$, see \cite[Lemma 6.3]{LN15}, \cite[Eq. (2.2)]{BBGZ13},  we can write
	\begin{equation}
	\label{eq: metric 1}
	E(\varphi)-E(u) = \int_0^1 \int_X (\varphi-u) H_m((1-t)u +t\varphi) \,dt. 	
	\end{equation}
	Set $w_t:=(1-t)u +tv$, for $t\in [0,1]$, and observe that 
	$$
	(1-t)u +t\varphi = \max(w_t,u) \  \ {\rm and}\ \ \Id_{\{w_t>u\}} = \Id_{\{v>u\}}, \ \forall t\in (0,1]. 
	$$
	It then follows from the plurifine locality that 
	$$
	\Id_{\{v>u\}}H_m(\max(w_t,u)) = \Id_{\{w_t>u\}}H_m(\max(w_t,u)) =  \Id_{\{v>u\}} H_m(w_t). 
	$$
	Using this, \eqref{eq: metric 1}, and the equality $\varphi-u=\Id_{\{v>u\}}(v-u)$, we can write
	\begin{equation*}
				E(\varphi)- E(u) = \int_0^1 \int_{\{v>u\}} (v-u) H_m(w_t)\, dt.
	\end{equation*}
	On the other hand, it follows from Corollary \ref{cor: Hes rooftop} that 
	\[
	H_m(P(w_t,v)) \leq {\bf 1}_{\{w_t\leq v\}} H_m(w_t) +  {\bf 1}_{\{w_t\geq v\}} H_m(v). 
	\]
	Using this, Corollary \ref{cor: derivative} and the fact that $\{w_t\leq  v\} = \{u\leq v\}$, for $t\in [0,1)$, we get
	\begin{flalign*}
	E(v) -E(\psi)&= \int_0^1 \int_X  (v-\min(u,v)) H_m(P(w_t, v)) \, dt \\ &\leq \int_0^1 \int_{\{u<v\}}  (v-u) H_m(w_t)\, dt,	
	\end{flalign*}
	hence the conclusion. 
\end{proof}

\begin{lemma}\label{lem: P monotonicity of distance}
For all $u,v,w\in \Eone$ we have $d(u,v) \geq d(P(u,w),P(v,w))$. 
\end{lemma}

\begin{proof}
	We first assume that $v\leq u$. It follows that $v\leq \max(v,P(u,w))\leq u$, hence by Lemma \ref{lem: basic properties}$(iii)$ and Proposition \ref{prop: inequality max P} we have
\begin{flalign*}
		d(v,u)& \geq d(v, \max(v,P(u,w))) \geq  d(P(u,w),P(P(u,w),v))\\
		& = d(P(u,w),P(v,w)). 
\end{flalign*}
Observe that the last identity follows from the fact and $P(P(u,w),v)=P(u,w,v)$ and $P(u,w,v)= P(w, v)$ since $v\leq u$.
Now, we remove the assumption $u\geq v$. Since $\min(u,v)\geq P(u,v)$ we can use the first step to write
$
d(u,P(u,v)) \geq d(P(u,w),P(u,v,w)),
$
and 
 $
 d(v,P(u,v))\geq d(P(v,w),P(u,v,w)).
 $ 
To finish the proof, it suffices to use Lemma \ref{lem: basic properties}(iii) and to note that 
$
P(P(u,w),P(v,w))=P(u,v,w).
$
\end{proof}

\begin{theorem}
	$d$ is a distance on $\Eone$.
\end{theorem}
\begin{proof}
The quantity $d$ is non-negative, symmetric and finite by definition. 
The fact that $d$ is non degenerate is a simple consequence of the domination principle.  Suppose $d(u,v)=0$. Lemma \ref{lem: basic properties}(iii) implies that $d(u,P(u,v))=d(v,P(u,v))=0$. Moreover, Lemma \ref{lem: basic properties}(iii) gives that $P(u,v) \geq u$ a.e. with respect to $H_m(P(u,v))$. By the domination principle, see \cite{DL15} (or Theorem \ref{thm: domination principle}), we obtain that $P(u,v) \geq u$, hence trivially $u=P(u,v)$. By symmetry $v=P(u,v)$, implying that $u=v$. 

It remains to prove the triangle inequality: for $u,v,\varphi\in \Eone$ we want to prove that
	$$
	d(u,v)\leq d(u,\varphi) +d(v,\varphi). 
	$$
	Using the definition of $d$  this amounts to showing that 
	$$
	 E(P(\varphi,u)) -E(P(u,v))\leq E(\varphi) -E(P(\varphi,v)). 
	$$
	But this follows from Lemma \ref{lem: P monotonicity of distance}, as we have the following sequence of inequalities: 
    \begin{flalign*}
    E(\varphi) - E(P(\varphi,v)) &= d(\varphi, P(\varphi, v)) 
    \geq  d( P(\varphi, u), P(P(\varphi, v), u))\\
     &=E(P(\varphi, u)) - E( P(\varphi, v, u))
    \geq   E(P(\varphi,u))-E(P(u,v)),
    \end{flalign*}
	where in the last line we have used the monotonicity of $E$, Lemma \ref{lem: basic properties}. 
\end{proof}

\subsection{Comparison with $I_1$}

\begin{lemma} \label{lem: halfwayest} For all  $u,v \in \mathscr E^1$ we have
$d\left(u,\frac{u+v}{2}\right) \leq \frac{3(m+1)}{2}d(u,v).$
\end{lemma}

\begin{proof} We have the following estimates:
\begin{flalign*}
d  &\Big(u,\frac{u + v}{2}\Big) = d\Big(u, P\Big(u,\frac{u + v}{2}\Big)\Big) + d\Big(\frac{u + v}{2},P\Big(u,\frac{u + v}{2}\Big)\Big)\\
& \leq d(u, P(u,v)) + d\Big(\frac{u+v}{2},P(u,v)\Big)\\
&\leq  \int_X (u - P(u,v))H_m(P(u,v)) + \int_X\Big(\frac{u+v}{2} - P(u,v)\Big) H_m(P(u,v))\\
&\leq \frac{3}{2}\int_X (u - P(u,v))H_m(P(u,v))+ \frac{1}{2}\int_X(v - P(u,v))H_m(P(u,v))\\
& \leq \frac{3(m+1)}{2}d(u, P(u,v)) + \frac{m+1}{2}d(v,P(u,v))\\
&\leq \frac{3(m+1)}{2}d(u,v),
\end{flalign*}
where in the second line we have additionally used that $P(u,v) \leq P(u,(u + v)/2)$.
\end{proof}

\begin{theorem}\label{thm: Darvas comparison}
For all $u,v \in \Eone$ we have 
$$
 d(u,v) \leq \int_X |u-v|(H_m(u)+H_m(v)) \leq  3(m+1) 2^{m+2}d(u,v).
$$
\end{theorem}

\begin{proof}
It follows from  Lemma \ref{lem: basic properties}  that 
$d(u, v) = d(u, P(u,v))+ d(v, P(u,v))$. 
Since the  energy $E$ is concave along affine curves, Proposition \ref{prop: basic fact on energy}, we have
\begin{eqnarray*}
d(u, P(u,v))&=& E(u)- E(P(u,v)) \leq  \int_X (u-P(u,v)) H_m(P(u,v))\\
&\leq &  \int_{\{v=P(u,v)\}} (u-v) H_m(v) \leq \int_X |u-v| H_m(v).
\end{eqnarray*}
Similarly we get $d(v, P(u,v))\leq \int_X |u-v| H_m(u) $. Putting these two inequalities together we get the first inequality.

Next we establish the lower bound for $d$. 
By  Lemma \ref{lem: halfwayest} and the Pythagorean formula we have
\begin{flalign*}
\frac{3(m+1)}{2}d(u,v) & \geq d\Big(u,\frac{u + v}{2}\Big) \geq d\Big(u,P\Big(u,\frac{u + v}{2}\Big)\Big)\\ &\geq \int_X \Big(u - P\Big(u,\frac{u + v}{2}\Big)\Big) H_m(u).
\end{flalign*}
By a similar reasoning as above, and the fact that 
$
2^m H_m((u + v)/2) \geq H_m(u)
$
 we can write:
\begin{flalign*}
\frac{3(m+1)}{2}d(u,v) &\geq d\Big(u,\frac{u + v}{2}\Big) \geq d\Big(\frac{u+v}{2},P\Big(u,\frac{u + v}{2}\Big)\Big)\\
&\geq \int_X \Big(\frac{u+v}{2} - P\Big(u,\frac{u + v}{2}\Big)\Big) H_m((u + v)/2)\\
&\geq \frac{1}{2^m} \int_X \Big(\frac{u+v}{2} - P\Big(u,\frac{u + v}{2}\Big)\Big) H_m(u).
\end{flalign*}
Adding the last two estimates we obtain
\begin{flalign*}
 3 (m+1)&  2^{m} d(u,v)  \\
 &  \geq  \int_X \Big(\Big(u - P\Big(u,\frac{u + v}{2}\Big)\Big)+\Big(\frac{u+v}{2} - P\Big(u,\frac{u + v}{2}\Big)\Big) \Big) H_m(u)\\
& \geq \frac{1}{2}\int_X |u - v|  H_m(u).
\end{flalign*}
By symmetry we also have $3(m+1) 2^{m+1} d(u,v) \geq \int _X |u - v| H_m(v)$, and adding these last two estimates together the lower bound for $d$ is established.
\end{proof}

\begin{lemma}\label{lem: control sup by d}
There exists $A, B\geq 1$ such that for any $\varphi\in \mathscr{E}^1$
$$
-d(0, \varphi)\leq \sup_X \varphi \leq A d(0, \varphi)+B.
$$
\end{lemma}
\begin{proof}
If $\sup_X \varphi \leq 0$, then the right-hand side inequality is trivial, while
$$
-d(0,\varphi)=E(\varphi) \leq \sup_X \varphi.
$$
We therefore assume that $\sup_X \varphi \geq 0$. In this case the left-hand inequality is trivial. By compactness property of the set of normalized $\omega$-$m$-sh functions \cite[Lemma 2.13]{Lu13} we have
$$
\int_X |\varphi-\sup_X \varphi|\omega^n  \leq C_1, 
$$
where $C_1>0$ is a uniform constant. Using Theorem \ref{thm: Darvas comparison} the result then follows in the following manner: 
\begin{flalign*}
d(0,\varphi) & \geq  C_2 I_1(0,\varphi)  \geq  C_2\int_X |\varphi|\omega^n \geq  C_2 \sup_X \varphi - C_2 \int_X |\varphi-\sup_X \varphi | \omega^n\\
& \geq C_2 \sup_X \varphi -C_1C_2. 
\end{flalign*}
\end{proof}

\subsection{$d$ is complete}

\begin{theorem}\label{thm: d is complete}
Assume that $u_j$ is a Cauchy sequence in $(\Eone,d)$. Then $u_j$ $d$-converges to $u\in \Eone$. In particular, we can extract a subsequence, still denoted by $u_j$, such that 
$$
\lim_{l\to +\infty}P(u_k,u_{k+1},....,u_{k+l}) \in \Eone.
$$
\end{theorem}
\begin{proof}
The argument is due to Darvas \cite{Dar15,Dar17AJM}, see also \cite[Theorem 3.10]{DDL3}.  
We can assume that 
$$ 
d(u_j,u_{j+1}) \leq 2^{-j}, j \geq 1. 
$$
As in the proof of \cite[Theorem 9.2]{Dar17AJM} we introduce the following sequences 
$$
\psi_{j,k}:=P(u_j, u_{j+1}, \dots,u_{k}), \ j\in \mathbb{N}, k\geq j. 
$$
Observe that, for $k\geq j+1$, $\psi_{j,k}=P(u_{j},  \psi_{j+1,k})$ and hence it follows from   Lemma \ref{lem: basic properties}(iii) and the triangle inequality that
\begin{eqnarray*}
d(u_j, \psi_{j,k})  &\leq d(u_j,  \psi_{j+1,k})
 \leq  d(u_j, u_{j+1}) +d(u_{j+1}, \psi_{j+1,k}) \\
 &\leq  2^{-j} + d( u_{j+1}, \psi_{j+1,k}).
\end{eqnarray*}
Repeating this argument several times we arrive at 
\begin{equation}
	\label{eq: completness 1}
	d(u_j,\psi_{j,k}) \leq 2^{-j+1}, \ \forall k\geq j+1.
\end{equation}
Using the triangle inequality for $d$ and the above we see that
$$
d(0,\psi_{j,k}) \leq  d(0,u_j) + d(u_j,\psi_{j,k}) \leq  d(0, u_1) + 2 + 2^{-j+1}
$$
is uniformly bounded. 
It follows from Theorem \ref{thm: Darvas comparison} and Lemma \ref{lem: control sup by d} that 
$I_1(0,\psi_{j,k})$, as well as $\sup_X \psi_{j,k}$, is uniformly bounded. We then infer, using the triangle inequality for $d$, that $d(0,\psi_{j,k}-\sup_X \psi_{j,k})$ is uniformly bounded hence so is $E(\psi_{j,k})$. Therefore, Proposition \ref{prop: continuity of I1} ensures that $\psi_j:=\lim_{k} \psi_{j,k} $ belongs to $\mathscr{E}^1$. From 
	\eqref{eq: completness 1} we obtain that $d(u_j,\psi_j) \leq 2^{-j+1}$, hence we only need to show that the $d$-limit of the increasing sequence $\{\psi_j\}_j  \subset \mathscr{E}^1$ is in $\mathscr E^1$.
  
Lemma \ref{lem: control sup by d} implies that $\sup_X \psi_j$ is uniformly bounded, hence $\psi:= \lim_j \psi_j \in \SHm$.	Now $\psi_j$ increases a.e. towards $\psi$, hence  $\psi \in \mathscr{E}^1$. Therefore by Proposition \ref{prop: continuity of I1} we have $I_1(\psi_j,\psi)\to 0$. It thus follows  from Theorem \ref{thm: Darvas comparison} that
$d(\psi_j,\psi) \rightarrow 0$.
\end{proof}

\subsection{$\omega$-$m$-subharmonic subextension}

In the previous sections, we easily adapted the arguments in \cite{DDL3}. These are necessary to derive the following result which is important in the sequel. 

\begin{theorem}\label{thm: compactness}
Assume that $u_j \in \Em$ satisfies $\sup_X u_j=0$ and 
$H_m(u_j) \leq A H_m(\psi)$, for some positive constant $A$ and some $\psi \in \SH_m(X,\omega) \cap L^{\infty}(X)$. Then $u_j \in \Eone$, and  a subsequence of $u_j$ $d$-converges to some $u\in \Eone$. In particular, we can extract a subsequence of $u_j$, still denoted by $u_j$, such that 
$$
\lim_{l\to +\infty} P(u_k,...,u_{k+l}) \in \Eone, \ \forall k.
$$
\end{theorem}
The result above is also new in the Monge-Amp\`ere case. It produces in particular a $\omega$-$m$-sh function lying below a suitably chosen subsequence of $(u_j)$.

\begin{proof}
We will use $C_1,C_2,...$ to denote uniform constants. 

We can assume that $-1\leq \psi \leq 0$ and $u_j$ converges in $L^1$ to $u\in \SHm$. By the Chern-Levine-Nirenberg inequality \cite[Corollary 3.18]{Lu13} we have that
$$
\int_X |u_j| H_m(u_j) \leq A\int_X |u_j| H_m(\psi) \leq C_1, \forall j.
$$
 It thus follows from Proposition \ref{prop: basic fact on energy} that $u_j \in \Eone$  and $|E(u_j)| \leq C_1$. 
Thus by \cite[Lemma 6.8]{LN15} we have 
$$
\int_X u_j^2 H_m(\psi)  \leq  2\int_0^{+\infty} t \capa_m(u_j<-t) dt \leq C_2
$$
is also uniformly bounded. Therefore, by the proof of \cite[Lemma 11.5]{GZbook} we have 
$
\int_X (u_j-u) H_m(\psi) \to 0. 
$
Define $\tilde{u}_k := (\sup(u_l, l\geq k))^*$. Then 
$$
|u_k-u|  = 2\max(u,u_k) -u-u_k \leq 2 (\tilde{u}_k-u) + u-u_k.
$$ 
Since $\tilde{u}_k$ decreases to $u$, it follows that
\begin{equation}\label{eq: compactness 2}
\int_X |u_j-u| H_m(u_j)  \leq A\int_X |u_j-u| H_m(\psi) \to 0. 
\end{equation}
We next claim that $H_m(u) \leq A H_m(\psi)$. The proof of this part is taken from \cite{Ceg98}, \cite{GZ07}. 
After extracting a subsequence we can assume that 
$$
\int_X |u_j-u|H_m(u_j) \leq 2^{-j}.
$$ 
We define $v_j := \max(u_j,u-1/j)$. Then $v_j$ converges in $m$-capacity to $u$. Hence by \cite[Theorem 3.9]{Lu13} $H_m(v_j)$ weakly converges to $H_m(u)$.  On the other hand we have 
$$
\int_{\{u_j \leq u-1/j\}} H_m(u_j) \leq j \int_X |u_j-u|H_m(u_j) \leq j 2^{-j} \to 0.
$$
 We thus have, for any positive continuous function $\chi$, 
\begin{eqnarray*}
\int_X \chi H_m(u) &=& \lim_{j\to +\infty} \int_X \chi H_m(v_j) \geq   \limsup_{j\to +\infty}  \int_{\{u_j>u-1/j\}} \chi H_m(u_j)\\
&\geq &  \limsup_{j\to +\infty} \int_X \chi H_m(u_j),
\end{eqnarray*}
where in the first inequality we have used Lemma \ref{lem: Identity principle}. 
But $H_m(u_j)$ and $H_m(u)$ have the same total mass, hence $H_m(u_j)$ weakly converges to $H_m(u)$ and therefore $H_m(u) \leq A H_m(\psi)$ as claimed. This together with \eqref{eq: compactness 2}  yields $I_1(u_j,u)\to 0$, hence by Theorem \ref{thm: Darvas comparison} we have $d(u_j,u) \to 0$. The last statement follows from Theorem \ref{thm: d is complete}. 
\end{proof}

\section{Complex Hessian equations with prescribed singularity}\label{sect: solutions}

Given a non-pluripolar positive measure $\mu$ and a model potential $\phi$ such that $\mu(X) = \int_X \theta_{\phi}^n>0$, we want to find $u \in \Ephi$ such that $H_m(u) =\mu$. 

The strategy is described in \cite{DDL4} which is inspired by the supersolution method of \cite{GLZJDG}. One constructs supersolutions of a well chosen family of equations and takes the lower envelope of supersolutions to get a solution. The main issue is to bound the supersolutions from below. To make the arguments of \cite{DDL4} work in Hessian setting we need a volume-capacity comparison of  the form : 
 $$
 \int_E f\omega^n \leq \left ( \capa_{\phi}(E)\right )^{1+\varepsilon},   \ E \subset X,
 $$
 for some $\varepsilon >0$. Here
$$
\capa_{\phi}(E) = \sup \left\{\int_E H_m(u) \setdef u \in \SHm, \ \phi-1 \leq u \leq \phi \right\}.
$$
 In the flat case where $\omega =dd^c \|z\|^2$ and $X =\Omega\subset \mathbb{C}^n$, it was conjectured by B{\l}ocki \cite{Bl05}  that $\SH_m(\Omega) \subset L^{q}(\Omega)$, for all $q<nm/(n-m)$.  If the compact manifold version of B{\l}ocki's conjecture holds then the $L^{\infty}$ estimate in \cite{DDL4} can be adapted in the Hessian setting giving solution for $L^p$ densities $p>n/m$. In the general case of non-$m$-polar measures the approach in \cite{DDL4} using Cegrell's method \cite{Ceg98} also breaks down in the Hessian setting. 
 
 Below, we will follow the main lines of \cite{DDL4} with several modifications. One of this is the use of the complete metric $d$ in $\Eone$ to construct subextensions of a $d$-converging sequence in $\Eone$. This procedure not only  replaces the relative $L^{\infty}$  estimate in \cite{DDL4} but also allows us to solve the complex Hessian equation directly without regularizing the measure $\mu$ by taking local convolution.

\subsection{Existence of solutions for bounded densities}

To explain the main ideas of the proof we first start with the case where $\mu = f\omega^n$ for some $0\leq f\in L^{\infty}(X,\omega^n)$, and $\phi = P[\alpha\phi_0]$, for some $\alpha\in (0,1)$ and $\phi_0 \in \SHm$. The general case, which is more involved and requires extra work, will be treated later.  

\begin{theorem}\label{thm: Hes eq}
Assume that $\phi =P[\alpha \phi_0]$, where $\alpha\in (0,1)$,  $\phi_0 \in \SH_m(X,\omega)$,  $0\leq f \in L^{\infty}(X)$ and  $\int_X H_m(\phi) = \int_X f\omega^n$. Then there exists $u\in \mathscr{E}_{\phi}(X,\omega,m)$ such that $H_m(u) = f\omega^n$. 
\end{theorem}

As shown in \cite{DDL2, DnL17}, in this case one can use the $\phi$-capacity to establish a $L^{\infty}$-estimate.  We propose, however,  in this section a different approach using the envelope which is interesting in its own right.

\begin{lemma}\label{lem: uniform estimate 1}
Fix $\alpha\in (0,1)$ and let $\phi_0$ be a $\omega$-$m$-sh function on $X$, normalized by $\sup_X \phi_0=0$. Assume that $u\in \SH_m(X,\omega)$ is less singular than $\alpha \phi_0$ and
$$
H_m(u) =f\omega^n, \ \sup_X u=0,
$$ 
where $f\in L^p(X,\omega^n), \ p>n/m$. Then, for a constant $C$ depending on $p,n,m$, $X,\omega$, $\alpha$, $\|f\|_p$,  we have
$$
u\geq \alpha \phi_0-C. 
$$
\end{lemma}

\begin{proof}
Set $b:= (1-\alpha)^{-1}$ and $v_b := P(bu- \alpha b\phi_0) \in \SH_m(X,\omega)$. From the assumption that $u$ is less singular than $\alpha \phi_0$ we deduce that $bu- b \alpha \phi_0$ is bounded from below, hence $v_b$ is bounded.  Then $b^{-1} v_b + \alpha \phi_0 \leq u$ with equality on  $D:= \{v_b= bu -\alpha b \phi_0\}$. Hence  by Proposition  \ref{prop: viscosity}, Lemma \ref{lem: multilinearity} and Proposition \ref{prop: orthogonal} we have 
$$
b^{-m}H_m(v_b) = {\bf 1}_D b^{-m} H_m(v_b) \leq  {\bf 1}_{D} H_m(b^{-1}v_b + \alpha \phi_0) \leq {\bf 1}_D H_m(u). 
$$
Next, we want to bound $\sup_X v_b$. Let $q$ be the conjugate of $p$: $\frac{1}{p} +\frac{1}{q}=1$. By Proposition \ref{prop: orthogonal} we have 
\begin{eqnarray*}
	\int_X |v_b|^{1/q} H_m(v_b) &=& \int_{D} |bu -(b-1) \phi_0|^{1/q} H_m(v_b)\\
	&\leq &  \int_X (|bu|+ |(b-1)\phi_0|)^{1/q} b^m f \omega^n.
\end{eqnarray*}
Using the H\"older inequality we see that the above term is uniformly bounded. Since $\int_X H_m(v_b)=1$ we infer that 
 $\sup_X v_b$ is uniformly bounded. We thus can invoke \cite{DK14}, \cite{Lu13} to obtain a uniform bound for $v_b$,
 hence $b u \geq \alpha b \phi_0 -C$.  This completes the proof. 
\end{proof}

Using the same idea we obtain the following estimate : 

\begin{lemma}\label{lem: uniform estimate 2}
Fix $a\in (0,1)$, $\phi_0\in \SHm$, $\sup_X \phi_0=0$. Assume that $u\in \Em$ satisfies
$$
H_m(u) \leq f\omega^n + a H_m(\phi_0), \ \sup_X u=0, 
$$ 
where $f\in L^p(X,\omega^n), p>n/m$.  Then, for a constant $C$ depending on $p,n,m$, $X$, $\omega$, $a$, $\|f\|_p$, we have
$$
u\geq a^{1/m}\phi_0 -C. 
$$
\end{lemma}

\begin{proof}
Fix a constant $b>1$ such that $(1-b^{-1})^m =a$, and set
$$
v_b := P(bu-(b-1)\phi_0), \ \ D:= \{v_b = bu -(b-1)\phi_0\}.
$$
It follows from Theorem \ref{thm: strict m positive} that $v_b \in \Em$.  Since  $b^{-1} v_b + (1-b^{-1}) \phi_0 \leq u$  with equality on $D$, by Proposition  \ref{prop: viscosity} and Lemma \ref{lem: multilinearity},  we have 
$$
\Id_D  \left ( b^{-m}H_m(v_b) + (1-b^{-1})^m H_m(\phi_0) \right)  \leq  \Id_D H_m(u). 
$$
 Using the above inequality, the assumption, and Proposition \ref{prop: orthogonal} we deduce that 
$$
H_m(v_b) = \Id_D H_m(v_b)  \leq b^m f \omega^n. 
$$
Having this, we can proceed as in the proof of Lemma \ref{lem: uniform estimate 1}. The details are left to the interested readers. 
\end{proof}

\begin{proof}[Proof of Theorem \ref{thm: Hes eq}]
We use the supersolution method of \cite{GLZJDG,DDL4}. 
\medskip 

\noindent {\bf Construction of supersolutions.}
 Fix $a\in (0,1)$ and solve, for each $k>0$ 
$$
H_m(u_k) := a{\bf 1}_{\{\phi \leq -k\}} H_m(\max(\phi,-k)) + c_k f\omega^n, 
$$
with $\ u_k \in \Em, \ \sup_X u_k=0.$ Here $c_k> 0$ is a constant ensuring that the two sides have the same total mass. The existence of the solution was proved in \cite{LN15}. Computing the total mass we see that  $c_k\searrow c(a)\geq 1$ defined by 
\begin{equation*}
a \left (1- \int_X H_m(\phi)\right ) +  c(a) \int_X H_m(\phi) =1. 
\end{equation*}
It follows from Lemma \ref{lem: uniform estimate 2} that, for a uniform constant $C_1$ depending on the fixed parameters (and also on $a$), 
$$
u_k \geq \phi - C_1.
$$
For each $l >0$ we define $\tilde{u}_{k,l} := P(\min(u_k,u_{k+1},...,u_{k+l}))$. Then by Corollary \ref{cor: Hes rooftop}, for $t>0$ fixed and $k>t$ we have 
$$
{\bf 1}_{\{\phi >-t\}} H_m(\tilde{u}_{k,l}) \leq c_k f\omega^n. 
$$ 
As $l\to +\infty$, $\tilde{u}_{k,l}$ decreases to a function $\tilde{u}_k \in \SHm$ such that $\phi-C_1\leq \tilde{u}_{k} \leq 0$.  Thus by Theorem \ref{thm: lsc of Hes measure} we have 
$$
{\bf 1}_{\{\phi>-t\}} H_m(\tilde{u}_{k}) \leq c_k f\omega^n.
$$
As $k\to +\infty$, $\tilde{u}_k$ increases a.e. to a function $\tilde{u}\in \SHm$ such that $\phi-C_1 \leq \tilde{u} \leq 0$ and by Theorem \ref{thm: lsc of Hes measure} we have 
$$
{\bf 1}_{\{\phi >-t\}} H_m(\tilde{u}) \leq c f \omega^n. 
$$
Letting $t\to +\infty$ we arrive at $H_m(\tilde{u}) \leq cf \omega^n$.

\medskip 

\noindent {\bf Envelope of supersolutions is a solution.}
The above analysis shows that for each $j\in \mathbb{N}$, there exists $w_j \in \SH_m(X,\omega)$ such that $\phi-C_j \leq w_j \leq 0$ and 
$$
H_m(w_j) \leq (1+2^{-j}) f\omega^n.
$$ 
Adding a constant we can assume that $\sup_X w_j=0$. By Lemma \ref{lem: uniform estimate 1} we have 
$$
w_j \geq \alpha \phi_0 -C,
$$
for a uniform constant $C$. 
For $k,l \in \mathbb{N}$, we set as above 
$$
\tilde{w}_{k,l} := P(\min(w_k,...,w_{k+l})).
$$
Then,  $\tilde{w}_{k,l}\geq \alpha \phi_0 -C$, for all $k,l$, hence $\tilde{w}_{k} :=  \lim_{l} \tilde{w}_{k,l} \in \SHm$. Since $H_m(\tilde{w}_{k,l}) \leq (1+2^{-k}) f\omega^n$ it follows from Theorem \ref{thm: lsc of Hes measure} and Theorem \ref{thm: dominated convergence} that  $H_m(\tilde{w}_{k,l})$ weakly converges to $H_m(\tilde{w}_{k})$. We thus have
$$
H_m(\tilde{w}_{k}) \leq (1+ 2^{-k}) f\omega^n, \ \ \tilde{w}_{k} \geq \alpha \phi_0 -C. 
$$
As $k\to +\infty$, $\tilde{w}_{k}$ increases a.e. to $\tilde{w}$. Again, it follows from Theorem \ref{thm: lsc of Hes measure} that  $H_m(\tilde{w}_{k})$ weakly converges to $H_m(\tilde{w})$, hence  $H_m(\tilde{w}) \leq f \omega^n$. Since $\tilde{w} \geq \alpha \phi_0 - C$, it follows from Theorem \ref{thm: monotonicity} that  $\int_X H_m(\tilde{w}) \geq \int_X f\omega^n$. We thus have equality, finishing the proof. 
\end{proof}

\subsection{Existence of solutions for non-$m$-polar measures}

\begin{theorem}\label{thm: existence general}
Assume that $\mu$ is a positive measure vanishing on $m$-polar sets, and $\phi$ is a model potential such that $\mu(X) = \int_X H_m(\phi)>0$. Then there exists a unique $u \in \Ephi$ such that $H_m(u) =\mu$. 
\end{theorem}

\begin{proof}
 
 It suffices to treat the case when $\mu \leq A H_m(\psi_0)$, for some constant  $A>0$ and some $\psi_0 \in \SH_m(X,\omega)$, with $-1\leq \psi_0 \leq 0$. The general case will follow by   a well-known projection argument due to Cegrell as shown in \cite{GZbook,DDL2}.  
 
 In the arguments below we use $C$ to denote various uniform constants. 
  
\medskip 

\noindent {\bf Construction of supersolutions. } For each $c>1$, we claim that there exists $u_c \in \SH_m(X,\omega)$ such that 
$$
P[u_c] \geq \phi, \  \text{and}\  H_m(u_c) \leq c \mu.
$$

To prove the claim, we fix $a\in (0,1)$ and solve, using \cite[Theorem 1.3]{LN15},  for each $k>0$ 
$$
H_m(u_k) := a{\bf 1}_{\{\phi \leq -k\}} H_m(\max(\phi,-k)) + c_k \mu, 
$$
with $\ u_k \in \mathscr{E}, \ \sup_X u_k=0$. Recall that $\mathscr{E} := \mathscr{E}(X,\omega,m)$ is the class of $\omega$-$m$-sh functions $u$ with full mass, $\int_X H_m(u)=1$.  Here $c_k> 0$ is a constant ensuring that the two sides have the same total mass. Computing the total mass we see that  $c_k\to c(a)\geq 1$ defined by 
\begin{equation}\label{eq: mass}
a \left (1- \int_X H_m(\phi)\right ) +  c(a) \int_X H_m(\phi) =1. 
\end{equation}
Fix $b>1$ such that $(1-b^{-1})^m=a$ and set 
$$
v_k := P(bu_k-(b-1)\max(\phi,-k)).
$$
Since $0 = P[u_k]$, it follows from  Corollary \ref{coro: subextension} (with $u,v \in \Em$  hence $P[u]=P[v]=0$) that $v_k \in \Em$. Setting
$
D_k := \{v_k = bu_k - (b-1) \max(\phi,-k)\}, 
$
it follows from Proposition \ref{prop: viscosity} that  
$$
{\bf 1}_{D_k} \left(b^{-m}H_m(v_k)  + (1-b^{-1})^m H_m(\max(\phi,-k))\right) \leq H_m(u_k).  
$$
By the choice of $b$ and by Proposition \ref{prop: orthogonal} we have  $H_m(v_k) \leq c_kb^m \mu$. By Proposition \ref{prop: orthogonal} again we have 
$$
\int_X |v_k| H_m(v_k) \leq \int_{D_k} |b u_k -(b-1)\max(\phi,-k)|  b^m c_k \mu \leq C,
$$
where the last estimate follows from \cite[Corollary 3.18]{Lu13}.  It thus follows that $\sup_X v_k$ is uniformly bounded. We can invoke Theorem \ref{thm: compactness}  to construct a subsequence, still denoted by $v_j$, such that for all $k$,
$$
\tilde{v}_k:= \lim_{l\to +\infty}P(v_k,v_{k+1},...v_{k+l}) \in \Eone.
$$ 
For each $k,l$ we define 
$$
\tilde{u}_{k,l} :=P(u_k,...,u_{k+l}) ; \  \tilde{u}_k := \lim_{l\to +\infty} \tilde{u}_{k,l}, \ \tilde{u} := \left (\lim_{k\to +\infty} \tilde{u}_k\right)^*.
$$
By the above construction we have that $u_k \geq b^{-1} v_k  + (1-b^{-1}) \phi$, hence
$$
\tilde{u}_k \geq b^{-1} \tilde{v}_k  + (1-b^{-1}) \phi.
$$
It thus follows from Lemma  \ref{lem: concavity of P} that $P[\tilde{u}_k] \geq \phi$, hence $P[\tilde{u}] \geq \phi$.  Fixing $t>0$, by Corollary \ref{cor: Hes rooftop} we have that, for all $k>t$, 
$$
\Id_{\{\phi>-t\}} H_m(\tilde{u}_{k,l}) \leq c_k \mu. 
$$
Since $\{\phi>-t\}$ is quasi-open, we can invoke Theorem \ref{thm: lsc of Hes measure} to obtain, letting $l\to +\infty$ and then $k\to +\infty$, 
$$
\Id_{\{\phi>-t\}} H_m(\tilde{u})  \leq c(a) \mu, 
$$
Letting $t\to +\infty$ we obtain $H_m(\tilde{u}) \leq c(a)\mu$. 
From \eqref{eq: mass} we see that $c(a)\to 1$ as $a\to 1$, hence $c(a)$ can be made arbitrarily near $1$. This proves the claim.

\medskip 

\noindent {\bf Envelope of supersolutions is a solution.}
The first step shows that for each $j\in \mathbb{N}$, there exists $w_j \in \SH_m(X,\omega)$ such that 
$$
\sup_X w_j =0,\  P[w_j]  \geq  \phi,\ \text{and}\ H_m(w_j) \leq (1+2^{-j})\mu. 
$$ 
It follows from Theorem \ref{thm: strict m positive} that there exists a constant $\lambda>1$ such that $P(\lambda\phi) \in \SH_m(X,\omega)$. Fix $b>1$ such that $b =(b-1)\lambda$. It follows from  
$$
P[w_j] \geq (1-b^{-1})\lambda \phi \geq (1-b^{-1})P(\lambda \phi)
$$
and Corollary \ref{coro: subextension 2} that 
$$
h_j := P(b w_j -(b-1) P(\lambda \phi)) \in \Em. 
$$
Moreover, it follows from Proposition \ref{prop: orthogonal} that 
$$
\int_X |h_j| H_m(h_j) \leq 2\int_X (|b w_j|+ (b-1) |P(\lambda \phi)|)b^{m} \mu \leq C,
$$
where the last estimate follows from the Chern-Levine-Nirenberg inequality \cite[Corollary 3.18]{Lu13}. It thus follows that $\sup_X h_j$ is uniformly bounded, as well as $E(h_j)$. As in the proof of the claim we can find a subsequence, still denoted by $h_j$, such that 
$$
\tilde{h}_k := \lim_{l\to +\infty} P(h_k,...,h_{k+l}) \in \Eone. 
$$ 
As in the first step we set 
$$
\tilde{w}_{k,l}:= P(w_k,...,w_{k+l}), \ \tilde{w}_k:= \lim_{l \to +\infty} \tilde{w}_{k,l}, \ \tilde{w} := \left (\lim_{k\to +\infty} \tilde{w}_k \right )^*. 
$$
By construction  we have 
$$
\tilde{w}_k \geq b^{-1} \tilde{h}_k + (1-b^{-1})P(\lambda \phi),
$$ 
hence $\tilde{w}_k \in \SHm$.  It follows  from Proposition \ref{prop: rooftop envelope 1} that $\tilde{w}_{k,l} \in  \Ephi$. By Corollary \ref{cor: Hes rooftop} we have 
$$
H_m(\tilde{w}_{k,l})  \leq (1+ 2^{-k}) \mu, \ \int_X H_m(\tilde{w}_{k,l}) \geq \mu(X).  
$$
By Theorem \ref{thm: dominated convergence}  we have that $H_m(\tilde{w}_{k,l})$ weakly converges to $H_m(\tilde{w}_k)$,  hence 
$$
 H_m(\tilde{w}_k) \leq (1+2^{-k})\mu, \ \int_X H_m(\tilde{w}_k) \geq \mu(X).
$$
By Theorem \ref{thm: dominated convergence} again we have $H_m(\tilde{w}) \leq \mu$ and  $\int_X H_m(\tilde{w}) \geq \mu(X)$, hence equality. 
\end{proof}

\subsection{Uniqueness}\label{sect: uniqueness}

To prove uniqueness,  as shown in \cite{DDL4}, one can follow closely the argument of S. Dinew \cite{DiwJFA09}.
We provide here a new proof using the orthogonal property of the envelopes. We hope that  this proof, which  is also new in the Monge-Amp\`ere case, will be useful in studying Monge-Amp\`ere type equations on non-K\"ahler manifolds. 

\begin{theorem}\label{thm: uniqueness}
Let $\phi$ be a model potential and let $u,v \in \Ephi$. If $H_m(u)=H_m(v)$ then $u-v$ is constant. 
\end{theorem}
\begin{proof}
We normalize $u,v$ by $\sup_X u=0$, $\sup_X v =0$. Set $\mu := H_m(u)=H_m(v)$.  It follows from Lemma \ref{lem: Dem inequality} that $w:= \max(u,v)$ satisfies $H_m(w) \geq \mu$ and $\int_X H_m(w) =\mu(X)$, hence $H_m(w) =\mu$. Thus,  we can assume that $u\leq v$. 
\medskip
 
\noindent
{\bf Step 1.} We first assume that $\mu$ is concentrated on $\{u=v\}$ \footnote{One can also invoke the domination principle.}.  Fix $b>1$ and set 
 $$
 \varphi_b := P(bu -(b-1)v), \ D:= \{\varphi_b = bu-(b-1)v\}.
 $$ 
 It follows from Theorem \ref{thm: strict m positive} that $\varphi_b \in \Ephi$. Since  
$
 b^{-1} \varphi_b + (1-b^{-1}) v \leq u, 
 $
 with equality on $D$, 
 it follows from Proposition \ref{prop: viscosity} that 
 \begin{equation}
 	\label{eq: uniqueness 1}
 	 \Id_D H_m(b^{-1} \varphi_b + (1-b^{-1}) v) \leq \Id_D H_m(u).
 \end{equation}
Combining this with the fact that $H_m(\varphi_b)$ is concentrated on $D$, and Lemma \ref{lem: multilinearity}, we arrive at 
 $$
b^{-m}  H_m(\varphi_b) = b^{-m} {\bf 1}_D H_m(\varphi_b) \leq \Id_{D}H_m(u).
 $$
 Writing $H_m(\varphi_b) = f_b \mu$, for some $0\leq f_b \in L^1(\mu)$, and using the mixed Hessian inequality (Lemma \ref{lem: mixed Hes ineq}),  and  multilinearity of the Hessian measure (Lemma \ref{lem: multilinearity}) we obtain 
 \begin{eqnarray}
 H_m(b^{-1} \varphi_b + (1-b^{-1}) v) & \geq &  \sum_{k=0}^m \binom{m}{k} b^{-k} (1-b^{-1})^{m-k} \omega_{\varphi_b}^k \wedge \omega_v^{m-k} \wedge \omega^{n-m} \nonumber\\
& \geq &\sum_{k=0}^m \binom{m}{k} b^{-k} (1-b^{-1})^{m-k} f_b^{k/m} \mu \nonumber \\
 &= &\left ( b^{-1}f_b^{1/m} + 1-b^{-1}\right )^m \mu. \label{eq: mixed 1}
\end{eqnarray}
 From  \eqref{eq: uniqueness 1} and \eqref{eq: mixed 1} we have 
 $$
\Id_D \left(b^{-1} f_b^{1/m} + 1-b^{-1}\right)^m \mu \leq \Id_D H_m(b^{-1} \varphi_b + (1-b^{-1}) v) \leq \Id_{D}\mu. 
 $$
We thus have $f_b\leq 1$, hence $f_b=1$, $\mu$-a.e. because $\int_X f_b \mu = \mu(X)$.   It thus follows from \eqref{eq: mixed 1} that, for $\psi_b:= b^{-1} \varphi_b + (1-b^{-1}) v$, we have $H_m(\psi_b) \geq \mu$ with the same total mass, hence $H_m(\psi_b)= \mu$. Thus, we have $\mu = H_m(\psi_b) =H_m(\varphi_b)$, therefore 
\begin{equation}
\label{eq: uniqueness 2}
\mu(\psi_b<u) =  \int_{\{\psi_b<u\}}H_m(\psi_b) = \int_{\{\varphi_b < bu-(b-1)v\}} H_m(\varphi_b) =0,
\end{equation}
where in the last equality we use the fact that $H_m(\varphi_b)$ is concentrated in the contact set  $\{\varphi_b = bu-(b-1)v\}$, thanks to Proposition \ref{prop: orthogonal}. 
Now, we use the assumption that $\mu$ is concentrated on $\{u=v\}$ to deduce, using \eqref{eq: uniqueness 2}, that $\mu$ is concentrated on the set $\{\varphi_b =u=v\}$. Therefore
\begin{equation}\label{eq: uniqueness bis}
\mu(X) = \mu(u=\varphi_b)  \leq \mu(u\leq \sup_X \varphi_b).
\end{equation}
From \eqref{eq: uniqueness bis} and the assumption that $\mu$ vanishes on $m$-polar sets, we infer that $\sup_X \varphi_b$ is uniformly bounded. Now,  letting $b\to +\infty$ we see that the function $\lim_{b\to +\infty}(\varphi_b-\sup_X \varphi_b)$ is a $\omega$-$m$-sh function which takes value $-\infty$ in the set $\{u<v\}$. This forces $\{u<v\}$ to be $m$-polar,  hence $u =v$.

\medskip

\noindent {\bf Step 2.} We treat the general case. We use the same notations and repeat the same arguments as above to  arrive at \eqref{eq: uniqueness 2}. We then get $H_m(\psi_b)=H_m(u) =\mu$, and $\psi_b \leq u$, and $\mu(\psi_b<u) =0$. Using the first step we have that $u=\psi_b$. Letting $b\to +\infty$ we obtain $u=v$.   
  \end{proof}

\subsection{Aubin-Yau equation}
Having the solutions to the complex Hessian equation $H_m(u)= \mu$, one can follow \cite{DDL2,DDL4} to prove the following result: 
\begin{theorem}
Assume that $\mu$ is a non-$m$-polar positive measure on $X$ and $\phi$ is a model potential. Then there exists a unique $u\in \Ephi$ such that $H_m(u) =e^u \mu$.  
\end{theorem}
We omit the proof of the above theorem and refer the interested  readers to \cite{DDL2,DDL4}.

\subsection{A Hodge index type inequality}\label{subsect: log-concave inequality}

The proof of Theorem \ref{thm: log concave} is very similar to that of \cite[Theorem 5.1]{DDL4} given Theorem \ref{thm: monotonicity intro}, Theorem \ref{thm: Hes eq intro} and the mixed Hessian inequality (Lemma \ref{lem: mixed Hes ineq}). For the reader's convenience we give the details below. 

\begin{proof}[Proof of Theorem \ref{thm: log concave}]
For each $j=1,...,m$, let $v_j \in \Em_{P[u_j]}$ solve $H_m(v_j) =c_j\omega^n$, where $c_j = \int_X H_m(v_j) =\int_X H_m(u_j)$. The existence of $v_j$ follows from Theorem \ref{thm: Hes eq intro}. The mixed Hessian inequality,  Lemma \ref{lem: mixed Hes ineq}, gives  
$$
H_m(v_1,...,v_m) \geq (c_1...c_m)^{1/m} \omega^n. 
$$
By Lemma \ref{lem: mass u Pu} we have that $\int_X H_m(v_1,...,v_m)=\int_X H_m(u_1,...,u_m)$, hence integrating the above inequality over $X$, we obtain the result. 
\end{proof}

\bibliographystyle{//Users/lu/Dropbox/Bib/amsplain_nodash.bst}  
\bibliography{//Users/lu/Dropbox/Bib/Biblio.bib}

\providecommand{\bysame}{\leavevmode\hbox to3em{\hrulefill}\thinspace}
\providecommand{\MR}{\relax\ifhmode\unskip\space\fi MR }
\providecommand{\MRhref}[2]{%
  \href{http://www.ams.org/mathscinet-getitem?mr=#1}{#2}
}
\providecommand{\href}[2]{#2}
\begin{thebibliography}{10}

\bibitem{AV10}
S.~Alesker and M.~Verbitsky, \emph{Quaternionic {M}onge-{A}mp\`ere equation and
  {C}alabi problem for {HKT}-manifolds}, Israel J. Math. \textbf{176} (2010),
  109--138. \MR{2653188}

\bibitem{Au78}
Thierry Aubin, \emph{\'{E}quations du type {M}onge-{A}mp\`ere sur les
  vari\'{e}t\'{e}s k\"{a}hl\'{e}riennes compactes}, Bull. Sci. Math. (2)
  \textbf{102} (1978), no.~1, 63--95. \MR{494932}

\bibitem{BT76}
Eric Bedford and B.~A. Taylor, \emph{The {D}irichlet problem for a complex
  {M}onge-{A}mp\`ere equation}, Invent. Math. \textbf{37} (1976), no.~1, 1--44.
  \MR{0445006}

\bibitem{BT82}
Eric Bedford and B.~A. Taylor, \emph{A new capacity for plurisubharmonic
  functions}, Acta Math. \textbf{149} (1982), no.~1-2, 1--40. \MR{674165}

\bibitem{BT87}
Eric Bedford and B.~A. Taylor, \emph{Fine topology, \v{S}ilov boundary, and
  {$(dd^c)^n$}}, J. Funct. Anal. \textbf{72} (1987), no.~2, 225--251.
  \MR{886812}

\bibitem{Berm18}
Robert~J. Berman, \emph{From {M}onge-{A}mp\`ere equations to envelopes and
  geodesic rays in the zero temperature limit}, Math. Z. \textbf{291} (2019),
  no.~1-2, 365--394. \MR{3936074}

\bibitem{BBGZ13}
Robert~J. Berman, S\'{e}bastien Boucksom, Vincent Guedj, and Ahmed Zeriahi,
  \emph{A variational approach to complex {M}onge-{A}mp\`ere equations}, Publ.
  Math. Inst. Hautes \'{E}tudes Sci. \textbf{117} (2013), 179--245.
  \MR{3090260}

\bibitem{Bl05}
Zbigniew B\l~ocki, \emph{Weak solutions to the complex {H}essian equation},
  Ann. Inst. Fourier (Grenoble) \textbf{55} (2005), no.~5, 1735--1756.
  \MR{2172278}

\bibitem{BEGZ10}
S\'{e}bastien Boucksom, Philippe Eyssidieux, Vincent Guedj, and Ahmed Zeriahi,
  \emph{Monge-{A}mp\`ere equations in big cohomology classes}, Acta Math.
  \textbf{205} (2010), no.~2, 199--262. \MR{2746347}

\bibitem{CNS85}
L.~Caffarelli, L.~Nirenberg, and J.~Spruck, \emph{The {D}irichlet problem for
  nonlinear second-order elliptic equations. {III}. {F}unctions of the
  eigenvalues of the {H}essian}, Acta Math. \textbf{155} (1985), no.~3-4,
  261--301. \MR{806416}

\bibitem{Ceg98}
Urban Cegrell, \emph{Pluricomplex energy}, Acta Math. \textbf{180} (1998),
  no.~2, 187--217. \MR{1638768}

\bibitem{Ch16}
Mohamad Charabati, \emph{Modulus of continuity of solutions to complex
  {H}essian equations}, Internat. J. Math. \textbf{27} (2016), no.~1, 1650003,
  24. \MR{3454681}

\bibitem{CW01}
Kai-Seng Chou and Xu-Jia Wang, \emph{A variational theory of the {H}essian
  equation}, Comm. Pure Appl. Math. \textbf{54} (2001), no.~9, 1029--1064.
  \MR{1835381}

\bibitem{CP19}
Tristan~C. Collins and Sebastien Picard, \emph{{ The Dirichlet problem for the
  $k$-Hessian equation on a complex manifold}}, arXiv:1909.00447 (2019).

\bibitem{Dar15}
Tam\'{a}s Darvas, \emph{The {M}abuchi geometry of finite energy classes}, Adv.
  Math. \textbf{285} (2015), 182--219. \MR{3406499}

\bibitem{Dar17AJM}
Tam\'{a}s Darvas, \emph{The {M}abuchi completion of the space of {K}\"{a}hler
  potentials}, Amer. J. Math. \textbf{139} (2017), no.~5, 1275--1313.
  \MR{3702499}

\bibitem{Dar17IMRN}
Tam\'{a}s Darvas, \emph{Metric geometry of normal {K}\"{a}hler spaces, energy
  properness, and existence of canonical metrics}, Int. Math. Res. Not. IMRN
  (2017), no.~22, 6752--6777. \MR{3737320}

\bibitem{Dar18S}
Tam\'{a}s Darvas, \emph{Geometric pluripotential theory on {K}\"{a}hler
  manifolds}, Advances in complex geometry, Contemp. Math., vol. 735, Amer.
  Math. Soc., Providence, RI, 2019, pp.~1--104. \MR{3996485}

\bibitem{DDL3}
Tam\'{a}s Darvas, Eleonora Di~Nezza, and Chinh~H. Lu, \emph{{$L^1$} metric
  geometry of big cohomology classes}, Ann. Inst. Fourier (Grenoble)
  \textbf{68} (2018), no.~7, 3053--3086. \MR{3959105}

\bibitem{DDL4}
Tam\'as Darvas, Eleonora Di~Nezza, and Chinh~H. Lu, \emph{{Log-concavity of
  volume and complex Monge-Amp\`ere equations with prescribed singularity}},
  arXiv:072018 (2018).

\bibitem{DDL2}
Tam\'{a}s Darvas, Eleonora Di~Nezza, and Chinh~H. Lu, \emph{Monotonicity of
  nonpluripolar products and complex {M}onge-{A}mp\`ere equations with
  prescribed singularity}, Anal. PDE \textbf{11} (2018), no.~8, 2049--2087.
  \MR{3812864}

\bibitem{DDL1}
Tam\'{a}s Darvas, Eleonora Di~Nezza, and Chinh~H. Lu, \emph{On the singularity
  type of full mass currents in big cohomology classes}, Compos. Math.
  \textbf{154} (2018), no.~2, 380--409. \MR{3738831}

\bibitem{DDL5}
Tam\'as Darvas, Eleonora Di~Nezza, and Chinh~H. Lu, \emph{{The metric geometry
  of singularity types}}, arXiv:1909.00839 (2019).

\bibitem{DDL6}
Tam\'as Darvas, Eleonora Di~Nezza, and Chinh~H. Lu, \emph{{Relative
  Pluripotential theory: a survey}}, Preprint (2020).

\bibitem{DR16}
Tam\'{a}s Darvas and Yanir~A. Rubinstein, \emph{Kiselman's principle, the
  {D}irichlet problem for the {M}onge-{A}mp\`ere equation, and rooftop obstacle
  problems}, J. Math. Soc. Japan \textbf{68} (2016), no.~2, 773--796.
  \MR{3488145}

\bibitem{DnL17}
Eleonora Di~Nezza and Chinh~H. Lu, \emph{Complex {M}onge-{A}mp\`ere equations
  on quasi-projective varieties}, J. Reine Angew. Math. \textbf{727} (2017),
  145--167. \MR{3652249}

\bibitem{NN13}
Nguyen~Quang Dieu and Nguyen~Thac Dung, \emph{Radial symmetric solution of
  complex {H}essian equation in the unit ball}, Complex Var. Elliptic Equ.
  \textbf{58} (2013), no.~9, 1261--1272. \MR{3170697}

\bibitem{DiwJFA09}
S{\l}awomir Dinew, \emph{Uniqueness in {$\mathcal E(X,\omega)$}}, J. Funct.
  Anal. \textbf{256} (2009), no.~7, 2113--2122. \MR{2498760}

\bibitem{DK14}
S{\l}awomir Dinew and S{\l}awomir Ko{\l}odziej, \emph{A priori estimates for
  complex {H}essian equations}, Anal. PDE \textbf{7} (2014), no.~1, 227--244.
  \MR{3219505}

\bibitem{DK17}
S{\l}awomir Dinew and S{\l}awomir Ko{\l}odziej, \emph{Liouville and
  {C}alabi-{Y}au type theorems for complex {H}essian equations}, Amer. J. Math.
  \textbf{139} (2017), no.~2, 403--415. \MR{3636634}

\bibitem{DK18}
S{\l}awomir Dinew and S{\l}awomir Ko{\l}odziej, \emph{Non standard properties
  of {$m$}-subharmonic functions}, Dolomites Res. Notes Approx. \textbf{11}
  (2018), no.~Special Issue Norm Levenberg, 35--50. \MR{3895934}

\bibitem{DL15}
S{\l}awomir Dinew and Chinh~H. Lu, \emph{Mixed {H}essian inequalities and
  uniqueness in the class {$\mathcal{E}(X,\omega,m)$}}, Math. Z. \textbf{279}
  (2015), no.~3-4, 753--766. \MR{3318249}

\bibitem{EG17}
Ayoub El-Gasmi, \emph{{The Dirichlet problem for the complex Hessian operator
  in the class $\mathcal{N}_m(H)$}}, arXiv:1712.06911.

\bibitem{EGZ11}
Philippe Eyssidieux, Vincent Guedj, and Ahmed Zeriahi, \emph{Viscosity
  solutions to degenerate complex {M}onge-{A}mp\`ere equations}, Comm. Pure
  Appl. Math. \textbf{64} (2011), no.~8, 1059--1094. \MR{2839271}

\bibitem{Gar59}
Lars G{\aa}rding, \emph{An inequality for hyperbolic polynomials}, J. Math.
  Mech. \textbf{8} (1959), 957--965. \MR{0113978}

\bibitem{GN18}
Dongwei Gu and Ngoc~Cuong Nguyen, \emph{The {D}irichlet problem for a complex
  {H}essian equation on compact {H}ermitian manifolds with boundary}, Ann. Sc.
  Norm. Super. Pisa Cl. Sci. (5) \textbf{18} (2018), no.~4, 1189--1248.
  \MR{3829745}

\bibitem{GLZJDG}
Vincent Guedj, Chinh~H. Lu, and Ahmed Zeriahi, \emph{{Plurisubharmonic
  envelopes and supersolutions}}, arXiv:1703.05254, to appear in J.
  Differential Geom. (2017).

\bibitem{GZ07}
Vincent Guedj and Ahmed Zeriahi, \emph{The weighted {M}onge-{A}mp\`ere energy
  of quasiplurisubharmonic functions}, J. Funct. Anal. \textbf{250} (2007),
  no.~2, 442--482. \MR{2352488}

\bibitem{GZbook}
Vincent Guedj and Ahmed Zeriahi, \emph{Degenerate complex {M}onge-{A}mp\`ere
  equations}, EMS Tracts in Mathematics, vol.~26, European Mathematical Society
  (EMS), Z\"{u}rich, 2017. \MR{3617346}

\bibitem{HLP16}
F.~Reese Harvey, H.~Blaine Lawson, Jr., and Szymon Pli\'{s}, \emph{Smooth
  approximation of plurisubharmonic functions on almost complex manifolds},
  Math. Ann. \textbf{366} (2016), no.~3-4, 929--940. \MR{3563228}

\bibitem{Hou}
Zuoliang Hou, \emph{Complex {H}essian equation on {K}\"{a}hler manifold}, Int.
  Math. Res. Not. IMRN (2009), no.~16, 3098--3111. \MR{2533797}

\bibitem{HMW}
Zuoliang Hou, Xi-Nan Ma, and Damin Wu, \emph{A second order estimate for
  complex {H}essian equations on a compact {K}\"{a}hler manifold}, Math. Res.
  Lett. \textbf{17} (2010), no.~3, 547--561. \MR{2653687}

\bibitem{Jbi}
Asma Jbilou, \emph{Complex {H}essian equations on some compact {K}\"{a}hler
  manifolds}, Int. J. Math. Math. Sci. (2012), Art. ID 350183, 48. \MR{3009566}

\bibitem{Kok10}
V.~N. Kokarev, \emph{Mixed volume forms and a complex equation of
  {M}onge-{A}mp\`ere type on {K}\"{a}hler manifolds of positive curvature},
  Izv. Ross. Akad. Nauk Ser. Mat. \textbf{74} (2010), no.~3, 65--78.
  \MR{2682372}

\bibitem{Kol98}
S{\l}awomir Ko{\l}odziej, \emph{The complex {M}onge-{A}mp\`ere equation}, Acta
  Math. \textbf{180} (1998), no.~1, 69--117. \MR{1618325}

\bibitem{KN16}
S{\l}awomir Ko{\l}odziej and Ngoc~Cuong Nguyen, \emph{Weak solutions of complex
  {H}essian equations on compact {H}ermitian manifolds}, Compos. Math.
  \textbf{152} (2016), no.~11, 2221--2248. \MR{3577893}

\bibitem{LT19}
Mau-Hai Le and Van-Dung Trieu, \emph{{Subextension of $m$-subharmonic
  functions}}, Vietnam Journal of Mathematics (2019).

\bibitem{Li04}
Song-Ying Li, \emph{On the {D}irichlet problems for symmetric function
  equations of the eigenvalues of the complex {H}essian}, Asian J. Math.
  \textbf{8} (2004), no.~1, 87--106. \MR{2128299}

\bibitem{Lu13}
Chinh~H. Lu, \emph{Solutions to degenerate complex {H}essian equations}, J.
  Math. Pures Appl. (9) \textbf{100} (2013), no.~6, 785--805. \MR{3125268}

\bibitem{Lu13JFA}
Chinh~H. Lu, \emph{Viscosity solutions to complex {H}essian equations}, J.
  Funct. Anal. \textbf{264} (2013), no.~6, 1355--1379. \MR{3017267}

\bibitem{Lu15}
Chinh~H. Lu, \emph{A variational approach to complex {H}essian equations in
  {$\Bbb{C}^n$}}, J. Math. Anal. Appl. \textbf{431} (2015), no.~1, 228--259.
  \MR{3357584}

\bibitem{LN15}
Chinh~H. Lu and Van-Dong Nguyen, \emph{Degenerate complex {H}essian equations
  on compact {K}\"{a}hler manifolds}, Indiana Univ. Math. J. \textbf{64}
  (2015), no.~6, 1721--1745. \MR{3436233}

\bibitem{NNC12}
Ngoc~Cuong Nguyen, \emph{Subsolution theorem for the complex {H}essian
  equation}, Univ. Iagel. Acta Math. (2013), no.~50, [2012 on articles],
  69--88. \MR{3235004}

\bibitem{NNC14}
Ngoc~Cuong Nguyen, \emph{H\"{o}lder continuous solutions to complex {H}essian
  equations}, Potential Anal. \textbf{41} (2014), no.~3, 887--902. \MR{3264825}

\bibitem{PPZ17}
Duong~H. Phong, Sebastien Picard, and Xiangwen Zhang, \emph{The {F}u-{Y}au
  equation with negative slope parameter}, Invent. Math. \textbf{209} (2017),
  no.~2, 541--576. \MR{3674222}

\bibitem{PPZ18}
Duong~H. Phong, Sebastien Picard, and Xiangwen Zhang, \emph{{Fu-Yau Hessian
  equations}}, arXiv:1801.09842 (2018).

\bibitem{PPZ19}
Duong~H. Phong, Sebastien Picard, and Xiangwen Zhang, \emph{On estimates for
  the {F}u--{Y}au generalization of a {S}trominger system}, J. Reine Angew.
  Math. \textbf{751} (2019), 243--274. \MR{3956695}

\bibitem{Pl13}
Szymon Pli\'{s}, \emph{{The smoothing of $m$-subharmonic functions}},
  arXiv:1312.1906 (2013).

\bibitem{RWN14}
Julius Ross and David Witt~Nystr\"{o}m, \emph{Analytic test configurations and
  geodesic rays}, J. Symplectic Geom. \textbf{12} (2014), no.~1, 125--169.
  \MR{3194078}

\bibitem{SA12}
A.~Sadullaev and B.~Abdullaev, \emph{Potential theory in the class of
  {$m$}-subharmonic functions}, Tr. Mat. Inst. Steklova \textbf{279} (2012),
  no.~Analiticheskie i Geometricheskie Voprosy Kompleksnogo Analiza, 166--192.
  \MR{3086763}

\bibitem{Sze18}
G\'{a}bor Sz\'{e}kelyhidi, \emph{Fully non-linear elliptic equations on compact
  {H}ermitian manifolds}, J. Differential Geom. \textbf{109} (2018), no.~2,
  337--378. \MR{3807322}

\bibitem{T95}
Neil~S. Trudinger, \emph{On the {D}irichlet problem for {H}essian equations},
  Acta Math. \textbf{175} (1995), no.~2, 151--164. \MR{1368245}

\bibitem{WN19}
David Witt~Nystr\"om, \emph{{Monotonicity of non-pluripolar Monge-Amp\`ere
  masses}}, Indiana University Mathematics Journal \textbf{68} (2019), no.~2,
  579--591.

\bibitem{Xia19}
Mingchen Xia, \emph{{Integration by parts formula for non-pluripolar product}},
  arXiv:1907.06359 (2019).

\bibitem{Xiao}
Jian Xiao, \emph{{Hodge-index type inequalities, hyperbolic polynomials and
  complex Hessian equations}}, arXiv:1810.04662 (2018).

\bibitem{Yau78}
Shing~Tung Yau, \emph{On the {R}icci curvature of a compact {K}\"{a}hler
  manifold and the complex {M}onge-{A}mp\`ere equation. {I}}, Comm. Pure Appl.
  Math. \textbf{31} (1978), no.~3, 339--411. \MR{480350}

\end{thebibliography}

\end{document}